\documentclass[11pt]{amsart}
\usepackage[T1]{fontenc}
\usepackage[utf8]{inputenc}
\usepackage[main=english]{babel}
\usepackage{geometry}
\usepackage[dvipsnames]{xcolor}
\usepackage{tikz}
\usepackage{amsmath, amssymb}
\usepackage{amsthm}
\usepackage{amsfonts}
\usepackage{mathtools} 
\usepackage[all]{xy}
\usepackage{booktabs,multirow,array}
\usepackage[unicode=true,pdfusetitle,
bookmarks=true,bookmarksnumbered=false,bookmarksopen=false,
breaklinks=false,backref=false,colorlinks=true]{hyperref}
\hypersetup{linkcolor=MidnightBlue, citecolor=PineGreen, urlcolor=PineGreen}

\usepackage{caption}     
\usepackage{csquotes}    
\usepackage[backend=biber,date=year, isbn=false, doi=false, url=false, eprint=true, style=alphabetic, maxnames=5]{biblatex}

\AtEveryBibitem{\clearlist{language}}
\theoremstyle{plain}
\newtheorem{thm}{\protect\theoremname}[section]
\newtheorem{lem}[thm]{\protect\lemmaname}
\newtheorem{prop}[thm]{\protect\propositionname}
\newtheorem{cor}[thm]{\protect\corollaryname}

\theoremstyle{definition}

\newtheorem{conj}[thm]{\protect\conjecturename}

\theoremstyle{remark}
\newtheorem{rem}[thm]{\protect\remarkname}

\providecommand{\corollaryname}{Corollary}
\providecommand{\definitionname}{Definition}
\providecommand{\lemmaname}{Lemma}
\providecommand{\propositionname}{Proposition}
\providecommand{\remarkname}{Remark}
\providecommand{\theoremname}{Theorem}
\providecommand{\conjecturename}{Conjecture}

\date{}

\DeclareMathOperator{\codim}{codim}

\usepackage{enumitem}

\definecolor{darkblue}{RGB}{0,0,139}   
\definecolor{negbase}{RGB}{40,90,200}     
\definecolor{sandbase}{RGB}{205,170,95}   
\definecolor{posbase}{RGB}{210,50,50}     
\colorlet{gridgray}{gray!40}           

\tikzset{
	cell/.style      = {fill=#1},           
	gridlines/.style = {draw=gridgray, very thin, step=0.5},
	boundary/.style  = {draw=black!75, thick},
	ppcoarse/.style = {draw=black!60, densely dotted, line width=1.1pt},  
	ppfine/.style   = {draw=black!60, densely dotted, line width=0.6pt},  
	pplabel/.style  = {anchor=north, inner sep=0pt, font=\scriptsize},
	every picture/.style = {scale=0.6, baseline=0pt},
}

\DeclareFieldFormat[article,book,incollection,misc]{title}{%
	\iffieldundef{doi}
	{\mkbibemph{#1}}
	{\href{https://doi.org/\thefield{doi}}{\mkbibemph{#1}}}}

\renewbibmacro{in:}{}
\DeclareFieldFormat{pages}{#1}
\renewbibmacro*{note+pages}{%
	\printfield{pages}%
	\setunit{\addcomma\space}%
	\printfield{note}%
	\newunit}

\setlist{nosep}

\title{Combinatorics behind discriminants of polynomial systems}
\author{Vladislav Pokidkin}
\address{HSE University}
\email{pokidkin.vladislav@gmail.com}
\thanks{The study has been funded within the framework of the HSE University Basic Research Program (HSE-BR-2025-061).}

\begin{document}
\begin{abstract}
	Polymatroids provide a suitable combinatorial framework for enumerating components of discriminants of polynomial systems, a generalization of classical discriminants introduced by Gelfand, Kapranov, and Zelevinsky for multivariate polynomials. We explore the relations among polymatroids, their induced matroids, and principal partitions, following works of Fujishige and Narayanan. We obtain a sufficient condition for a polymatroid contraction to commute with taking the induced matroid. We establish a correspondence between different classes of polymatroids and critical values of their principal partitions. Three new conjectures about discriminants of polynomial systems are formulated. 
\end{abstract}
\maketitle

\section{Introduction}
\subsection*{Discriminants and polymatroids}
This work is intended to provide a combinatorial background for the paper \cite{pokidkin_components_2025} on the geometry of discriminants of general systems of equations.

For a finite set $A$ in the lattice $\mathbb{Z}^n$, called the \textit{support set}, a general polynomial is a formal linear combination of monomials $\sum_{a\in A} c_a x^a$, where $c_a$ are independent variables. Polynomial systems of equations supported on $A_1,\dots,A_k$ form the space $\mathbb{C}_{A_{1}}\oplus \dots\oplus\mathbb{C}_{A_{k}}$, in which every point $c$ is identified with the tuple of general polynomials $\sum_{a\in A_i} c_{i,a} \;x^a$. The combinatorics of support sets is reflected in the structure of solutions of generic polynomial systems \cite{bernshtein_number_1975,khovanskii_newton_1978,esterov_multivariate_2016,khovanskii_newton_2016}, the intersection theory \cite{kaveh_mixed_2010},
the theory of resultants \cite{sturmfels_newton_1994,jensen_computing_2013}, and discriminants \cite{gelfand_discriminants_2008,esterov_newton_2010,cattani_mixed_2013,esterov_discriminant_2013,esterov_galois_2019,dickenstein_iterated_2023}.

The \textit{discriminant} in this space $\mathbb{C}_{A_{1}}\oplus \dots\oplus\mathbb{C}_{A_{k}}$ is the closure of the set of all polynomial systems $(f_1,\dots,f_k)$ having a \textit{degenerate root} $x$ in the algebraic torus $(\mathbb{C}^\times)^n$, i.e. a point at which $f_1(x)=\dots=f_k(x)=0$ and $df_1(x) \wedge \dots \wedge df_k(x)=0$.

In the 1990s, Gelfand, Kapranov, Zelevinsky, and Sturmfels initiated the study of such discriminants with motivations ranging from hypergeometric functions and PDEs to symbolic computations. A central problem in this study is to enumerate irreducible components of the discriminant and to describe their codimensions, multiplicities, and degrees. The simplest case $k>n$, when the discriminant coincides with the resultant, was completely resolved in \cite{sturmfels_newton_1994}. For $k=1$, the discriminant of one equation is projectively dual to a toric variety as noted in \cite{gelfand_discriminants_2008}, and its codimension is still not entirely understood \cite{furukawa_combinatorial_2021,cattani_non-splitting_2022}.

For the intermediate cases $1<k<n$, there are only partial results so far \cite{esterov_newton_2010,esterov_discriminant_2013,dickenstein_iterated_2023}. For the case $k=n$, the first progress is due to Esterov \cite{esterov_galois_2019}, who was motivated by the study of Galois groups of general systems of equations. The case $k=n$ is settled in \cite{pokidkin_components_2025}: the components of the discriminant are enumerated, their codimensions and degrees are computed. The description and some proofs are based on the theory of polymatroids.

Polymatroids were introduced by Edmonds in 1970 \cite{edmonds_submodular_2003}. An (integer)
\textit{polymatroid} $P$ is a pair $(E,f)$ consisting of a finite \textit{ground set} $E$ and a rank function $f:2^{E}\rightarrow\mathbb{Z}_{\geq0}$, which is submodular, monotone, and normalized ($f(\varnothing)=0$). For example, a collection of support sets $A_1,\dots,A_k$ containing the origin defines a (realizable) polymatroid $([k],f)$ with the rank function $f(I)=\dim \sum_{i\in I} \langle A_i\rangle $, where $\langle A_i\rangle$ is the sublattice generated by the finite set $A_i$. The theory of polymatroids is actively developing \cite{csirmaz_cyclic_2020,bonin_natural_2023,swartz_polymatroids_2024,crowley_bergman_2026}, and it intertwines with algebraic geometry \cite{pagaria_hodge_2023,eur_intersection_2024,baker_representation_2025}. 

This text has several goals:

\begin{enumerate}[itemsep=\medskipamount, topsep=\medskipamount,	parsep=0pt,	partopsep=0pt]
	\item We present the combinatorial machinery that was used in \cite{pokidkin_components_2025} as a black box for the case $k=n$. For the convenience of the reader, we give here a coherent, self-contained presentation of everything that we need, and relate it to important basic characteristics of polymatroids.
	\item We introduce and characterize different classes of polymatroids with a view towards enumerating components of the discriminant of systems of $k<n$ equations. We develop the link between polymatroids, their induced matroids, and principal partitions. The results allow us to anticipate the combinatorics behind discriminants and to formulate conjectures.
\end{enumerate}

Inspired by the theory of discriminants, we introduce the following definitions for polymatroids. The \textit{defect} of a subset $I$ of the ground set $E$ is the difference $\delta(I)=f(I)-|I|$. A \textit{proper} subset of a set is a subset that is neither the set nor the empty set. A set is \textit{essential} if every proper subset has a strictly greater defect than the set itself. A set is \textit{independent} if the defects of all its subsets are non-negative. An independent set is \textit{irreducible} if the defects of all its proper subsets are positive. A \textit{BK-set} is an independent set with zero defect. A \textit{BK-union} is a union of BK-sets.

The notion of essential sets arose from the work of Sturmfels on resultants \cite{sturmfels_newton_1994}. The idea of BK-sets and independent sets arose from the Kouchnirenko-Bernshtein theorem \cite{bernshtein_number_1975} and the Minkowski criterion for the vanishing of mixed volumes. Generic polynomial systems with an independent polymatroid have a non-empty set of solutions. The notion of irreducible sets arose from Esterov's conjecture on discriminants \cite{esterov_galois_2019}.

We characterize polymatroids via their ground sets. For example, a dependent polymatroid is a polymatroid such that the ground set is dependent, and a BK-polymatroid is a polymatroid such that the ground set is a BK-set. The theory of discriminants provides us with six classes of polymatroids: dependent, independent, essential, BK-unions, BK, and irreducible.

\subsection*{Polymatroids and the induced matroids}
Every polymatroid $P=(E,f)$ induces the matroid $M=(E,rk)$ \cite{mcdiarmid_independence_1973,welsh_matroid_1976}
such that 
\[
rk(I)=\min\,\{f(J)+|I\backslash J|,\;J\subseteq I\}.
\]

According to Corollary 11.1.2 of \cite{oxley_matroid_2011}, (in)dependent subsets in a polymatroid are (in)dependent subsets of the induced matroid. This means that the classes of (in)dependent polymatroids are uniquely determined by the induced matroids. Recall that a \textit{cycle} (\textit{cyclic subset}) in a matroid is a non-empty union of circuits. Induced matroids motivate us to introduce one more class of polymatroids: cyclic.

A \textit{polymatroid loop} is a singleton of zero polymatroid rank. In this work, we consider only loopless polymatroids. Under this assumption, the induced matroid is loopless as well. Polymatroids on singletons are considered as essential and irreducible polymatroids even though they do not have proper subsets.

We provide our own proofs for the following facts:
\begin{enumerate}[label=\textit{\alph*})]
	\item bases of cycles are BK-sets (Corollary \ref{Corollary. In cyclic, bases are BK});
	\item a subset is cyclic if and only if it is essential and has negative defect (Corollary \ref{Corollary. Cyclic=essential});
	\item bases of the induced matroid are bases of the polymatroid (Theorem \ref{Theorem. f(B)=f(E)});
	\item a set is a BK-union if and only if the polymatroid rank of the set equals its induced matroid rank (Corollary \ref{Corollary. rk(E)=f(E) iff BK-union}).
\end{enumerate}
\begin{thm}[\ref{Theorem. Contraction by a union of BK-sets}]
	For a polymatroid, the contraction by a BK-union commutes with taking the induced matroid.
\end{thm}
For polynomial systems on support sets $A_1,\dots,A_k$ with a dependent polymatroid, this means that the discriminant is the resultant of the subsystem indexed by the maximal cycle of the induced matroid. The study of contraction was motivated by evaluations of discriminants and by substitutions of solutions of a subsystem into the complementary subsystem.

For polynomial systems with a dependent polymatroid, the characterization of their discriminant is reduced to the study of the induced matroid and the resultant. For polynomial systems with an independent polymatroid, the induced matroid  is free, which does not provide sufficient information to study discriminants. In that case, the principal partition of the polymatroid is the first candidate to provide the necessary combinatorial framework for describing components of the discriminant.

\subsection*{Polymatroid principal partitions}
In 1968, Kishi and Kajitani introduced the concept of principal partitions in graph theory \cite{kishi_maximally_1968}, and later it was extended to matrices, matroids, submodular systems, and pairs of polymatroids \cite{fujishige_theory_2009}. The theory of principal partitions of submodular functions is well developed in the monographs of Fujishige \cite{fujishige_submodular_2005} and Narayanan \cite{narayanan_submodular_2009}.

A \textit{partition} of a finite set $E$ is a set of non-empty disjoint subsets of $E$, called \textit{blocks}, whose union is $E$. For a submodular function $s:2^E\rightarrow\mathbb{R}$ on a finite set $E$, a \textit{minimizer} is a subset at which the function $s$ attains its minimum. The set of all minimizers is known to form a distributive lattice of subsets with union and intersection as the lattice operations. By Birkhoff's representation theorem \cite{birkhoff_rings_1937,fujishige_theory_2009}, such a distributive lattice defines a partition of the ground set whose blocks are indexed by a poset, in such a way that every set in the lattice is the union of blocks enumerated by some order ideal (down-set).

For a real parameter $\lambda$, the lattice of minimizers of $s(I)-\lambda|I|$ consists of a single element for all but finitely many $\lambda$, called \textit{critical values}. All possible minimizers of the function $s(I)-\lambda|I|$ for all real $\lambda$ form a distributive lattice, which in turn determines the \textit{principal partition} of the submodular function $s$.

According to Narayanan \cite{narayanan_submodular_2009}, a polymatroid is \textit{molecular} if its rank function has exactly one critical value, and the polymatroid is $atomic$ if the principal partition is trivial. In Subsection \ref{Subsection. c.v.}, we show that the different classes of polymatroids are determined by the critical values of their defect functions.

\begin{thm}\label{Theorem. cv}
	Consider a polymatroid and the critical values of its defect function. Then
	\begin{enumerate}[label=\alph*\textup{)}]
		\item  the critical values are negative if and only if the polymatroid is cyclic \textup{(Proposition \ref{Proposition. cv cyclic polymatroid})};
		\item the critical values are non-negative if and only if the polymatroid is independent \textup{(Proposition \ref{Proposition. cv independent polymatroid})};
		\item the critical values are non-positive if and only if the ground set is a BK-union \textup{(Proposition \ref{Proposition. cv BK-union})};
		\item the polymatroid is molecular with zero critical value if and only if the polymatroid is BK \textup{(Theorem \ref{Theorem. cv BK-polymatroid})};
		\item the polymatroid is atomic with zero critical value if and only if the polymatroid is irreducible BK \textup{(Proposition \ref{Proposition. cv irreducible BK-polymatroid})};
		\item the polymatroid is atomic with the critical value $-1/|E|$ if and only if the ground set is a circuit \textup{(Proposition \ref{Proposition. cv circuit})};
		\item if the polymatroid is essential independent, then the polymatroid is atomic \textup{(Proposition \ref{Proposition. cv essential independent})}.
	\end{enumerate}
\end{thm}

\begin{table}[t]
	\caption{Classes of loopless polymatroids, the behavior of the critical values of the defect function, and the resulting structure of the discriminant.}
	\label{tab:classes}
	\centering
	\begin{tabular}{@{}>{\raggedright\arraybackslash}p{0.21\textwidth}
			>{\raggedright\arraybackslash}p{0.23\textwidth}
			>{\raggedright\arraybackslash}p{0.13\textwidth}
			>{\raggedright\arraybackslash}p{0.36\textwidth}@{}}
		\toprule
		Polymatroid class & Critical values of the defect & Proof & Discriminant \\
		\midrule
		Dependent & at least one negative &Prop. \ref{Proposition. cv independent polymatroid} &
		\multirow[t]{5}{=}{\raggedright resultant of the subsystem indexed by
			the maximal cycle of the induced matroid \cite{sturmfels_newton_1994}} \\
		\cmidrule(r){1-3}
		Dependent BK-union & all non-positive, at least one negative &Prop. \ref{Proposition. cv BK-union} & \\
		\cmidrule(r){1-3}
		Essential dependent ($=$ cyclic) & all negative & Prop. \ref{Proposition. cv cyclic polymatroid}& \\
		\cmidrule(r){1-3}
		Molecular dependent & a single value $\delta(E)/|E|<0$ &Prop. \ref{Proposition. cv independent polymatroid} \& Def.& \\
		\cmidrule(r){1-3}
		Circuit (atomic) & a single value $-1/|E|$ &Prop. \ref{Proposition. cv circuit}& \\
		\midrule
		BK-polymatroid (molecular) & a single value $0$ &Thm. \ref{Theorem. cv BK-polymatroid} &
		components of codimension one or two, enumerated by the principal partition	(Theorem~\ref{Theorem. Components for the discriminant with BK-polymatroid}) \\
		\addlinespace
		Irreducible BK-polymatroid (atomic) & a single value $0$ &Prop. \ref{Proposition. cv irreducible BK-polymatroid} &
		irreducible variety of codimension one or two (Esterov's conjecture,	proved in \cite{pokidkin_components_2025}; Theorem \ref{Theorem. Esterov conjecture}) \\
		\midrule
		Essential independent (atomic) &
		a single value $\delta(E)/|E|\geq 0$ &Prop. \ref{Proposition. cv essential independent} &
		irreducible variety (Conjecture~\ref{Conjecture. Atomic polymatroids}) \\
		\addlinespace
		
		Atomic independent &
		a single value $\delta(E)/|E|\geq 0$ &Prop. \ref{Proposition. cv independent polymatroid} \& Def. &
		irreducible variety (Conjecture~\ref{Conjecture. Atomic polymatroids}) \\
		\addlinespace
		Molecular independent &
		a single value $\delta(E)/|E|\geq 0$ &Prop. \ref{Proposition. cv independent polymatroid} \& Def. &
		components, enumerated by the principal
		partition (Conjecture~\ref{Conjecture. Independent polymatroids}) \\
		\cmidrule(r){1-3}
		Independent & all non-negative &Prop. \ref{Proposition. cv independent polymatroid} &
		components, enumerated by the principal
		partition (Conjecture~\ref{Conjecture. Independent polymatroids}) \\
		\bottomrule
	\end{tabular}
\end{table}

\subsection*{Principal partitions and discriminants}
In 2019, Esterov made the conjecture \cite{esterov_galois_2019}:
\begin{thm}\label{Theorem. Esterov conjecture}
	\textup{\cite{pokidkin_components_2025}} The discriminant of a polynomial system with an atomic BK-polymatroid is irreducible.
\end{thm}
The proof of this conjecture rests precisely on the fact that the empty set and the ground set are the only minimizers of the defect function. The proof also requires the polymatroid partition of a vector space from Proposition \ref{Proposition. Polymatroid partition of a vector space} of this paper.
\begin{conj}\label{Conjecture. Atomic polymatroids}
	The discriminant of a polynomial system with an atomic polymatroid is irreducible.
\end{conj}
This conjecture holds for all atomic polymatroids with a negative critical value (resultants), zero critical value (Esterov's theorem), and polymatroids on a singleton (projectively dual to toric varieties). The proof is expected to be similar to the proof for the case of BK-polymatroids. The codimension of discriminants with an atomic polymatroid can be arbitrary.

A \textit{stratum} of a discriminant is an irreducible algebraic subset that is distinguished as follows.
\begin{thm}
	\textup{\cite{pokidkin_components_2025}} For polynomial systems with a BK-polymatroid, the discriminant is a union of strata enumerated by blocks of the polymatroid principal partition.
\end{thm}
\begin{conj}
	For polynomial systems with an independent polymatroid, the discriminant is a union of strata enumerated by blocks of the polymatroid principal partition.
\end{conj}
The theorem and conjecture provide an attainable upper bound on the number of components of a discriminant. Strata do not enumerate components, and the polymatroid alone does not fully characterize components of discriminants in the case $k\leq n$. We need to decorate the polymatroid with some extra combinatorial data, such as the distinction between linear and nonlinear systems.

For a BK-polymatroid, BK-sets define the principal partition of the polymatroid and form a distributive lattice. For a collection of support sets $A_1,\dots,A_n$ corresponding to a BK-polymatroid on the finite set $[n]$, every \textit{join-irreducible} set $I$  in the distributive lattice of BK-sets is a BK-set that is not a union of its proper BK-subsets. The BK-set $I$ contains a unique maximal proper BK-subset $I'$, and the complement $I\backslash I'$ is a block of the principal partition. Let us call $I$ \textit{prelinear} if the projections of the sets $A_i,\;i \in I \backslash I'$, along the affine span of $\sum_{i\in I'} A_i$, can be shifted into the standard simplex by a lattice automorphism. Then the main result in \cite{pokidkin_components_2025} is 
\begin{thm}\label{Theorem. Components for the discriminant with BK-polymatroid}
	For support sets $A_1,\dots,A_n$ corresponding to a BK-polymatroid, the discriminant has components:
	\begin{enumerate}[label=\alph*\textup{)}]
		\item codimension-one components for nonlinear join-irreducible BK-sets;
		\item codimension-two components for prelinear join-irreducible BK-sets that are not contained in a nonlinear BK-set.
	\end{enumerate}
\end{thm}
Each component can be described constructively. To enumerate the components, we need to know an unexpectedly small amount of combinatorial information about the support sets. Even less data is needed: it suffices to consider the poset of the principal partition together with the bicoloring of its elements marking each join-irreducible BK-set as prelinear or nonlinear. By Proposition \ref{Proposition. Bicoloring}, the poset and coloring can be arbitrary. The poset underlies the optimal algorithm for solving such polynomial systems, described in \cite{brysiewicz_solving_2021} and implemented in the Macaulay2 package $\mathtt{DecomposableSparseSystems}$.

For polynomial systems with an arbitrary independent polymatroid, the hypersurface components of the discriminant correspond to some blocks with the minimal critical value of the polymatroid according to Theorem 2.31 of \cite{esterov_newton_2010}. The remaining blocks of the polymatroid principal partition are expected to correspond to strata of the discriminant of higher codimension. We conjecture that it is necessary to construct some weight function on the poset of the polymatroid principal partition to determine which higher-codimensional strata are components.

\begin{conj}\label{Conjecture. Independent polymatroids}
	For polynomial systems with an independent polymatroid, a stratum is a component of the discriminant if it is not covered by a stratum of less codimension with respect to the partial order on blocks of the polymatroid principal partition.
\end{conj}
The conjecture holds for polymatroids with a BK-polymatroid by Theorem \ref{Theorem. Components for the discriminant with BK-polymatroid}.

The structure of the paper is as follows. Sections 2-4 are devoted to the link between polymatroids and their induced matroids for bases, cycles, and ranks. This preliminary work is applied in Section 5 to establish the consistency of a polymatroid contraction with the induced matroid contraction. Principal partitions are introduced in Section 6. The same section contains the description of critical values for different classes of polymatroids (see Table \ref{tab:classes}). The last Section 7 is devoted to some results about realizable polymatroids.

\section{Bases of the induced matroid}
\begin{prop}[\cite{oxley_matroid_2011}, Proposition 1.1.6] \label{Proposition. Oxley Circuits}
	Suppose $I$ is an independent set in a matroid and $e$ is an element	such that $I\cup e$ is dependent; then there exists a unique circuit contained in $I\cup e$, and this circuit contains $e$.
\end{prop}
\begin{lem}	\label{Lemma. In essential, f(E)=f(E-e)}
	For an essential polymatroid and an element $e$ of the ground set $E$ with cardinality at least two, the equality $f(E)=f(E\backslash e)$ holds.
\end{lem}
\begin{proof}
	Since the ground set is essential, $\delta(E\backslash e)>\delta(E)$
	by definition. This is equivalent to $1>f(E)-f(E\backslash e)$. Since the function $f$ is monotone and integer, we obtain the equality.
\end{proof}
\begin{lem}\label{Lemma. Value of f equal}
	Suppose $I$ is an independent set in the matroid induced by a polymatroid and $e$ is an element such that $I\cup e$ is dependent. Then the polymatroid ranks are equal: $f(I)=f(I\cup e).$
\end{lem}
\begin{proof}
	By Proposition \ref{Proposition. Oxley Circuits}, there exists a
	circuit $C\subseteq I\cup e$ containing $e$.
	Notice that circuits are essential, and $f(C)=f(C\backslash e)$ by Lemma \ref{Lemma. In essential, f(E)=f(E-e)}. The polymatroid rank $f$ is submodular, and
	\begin{align*}
		f(C)+f(I) & \geq f(I\cup C)+f(I\cap C)=f(I\cup e)+f(C\backslash e),\\
		f(I) & \geq f(I\cup e).
	\end{align*}
	Since the polymatroid rank $f$ is monotone, we have the equality
	$f(I)=f(I\cup e).$
\end{proof}
\begin{cor}
	Circuits of the induced matroid have defect $-1$.
\end{cor}
\begin{proof}
	The defect of a circuit $C$ is negative because circuits are dependent. For every element $e\in C$, the complement $C\backslash e$ has a non-negative defect as an independent subset. Hence the defect of the circuit is $-1$.
\end{proof}
\begin{cor} \label{Corollary. Basis change}
	Bases of the induced matroid have the same defect and the same polymatroid rank.
\end{cor}
\begin{proof}
	Recall the basis-exchange property for matroids: if bases $B$ and
	$B'$ are distinct, and $e\in B\backslash B'$, then there exists
	an element $e'\in B'\backslash B$ such that $(B\cup e')\backslash e$ is a basis. By Lemma \ref{Lemma. Value of f equal}, a basis exchange does not change the polymatroid rank: $f(B)=f(B\cup e')=f((B\cup e')\backslash e)$.
\end{proof}
\begin{thm} \label{Theorem. f(B)=f(E)}
	All subsets containing a basis of the induced matroid have the same polymatroid rank. In particular, $f(B)=f(E)$ for a basis $B$ of the ground set $E$.
\end{thm}
\begin{proof}
	For every element $e\in E\backslash B$, the set $B\cup e$ is dependent, and we have the equality of polymatroid ranks $f(B)=f(B\cup e)$ by Lemma \ref{Lemma. Value of f equal}. For a subset $X\subseteq E\backslash(B\cup e)$, we have the submodular inequality $f(B\cup X\cup e)\leq f(B\cup X)+f(B\cup e)-f(B)=f(B\cup X)$. Since the polymatroid rank $f$ is monotone, the equality $f(B\cup X\cup e)=f(B\cup X)$ holds. Adding elements of $E\backslash B$ one by one to the basis $B$, we obtain $f(B)=f(E)$.
\end{proof}
In our setting, a \textit{polymatroid basis} is any subset $I$ achieving the rank of a polymatroid, $f(I)=f(E)$. Such subsets can have different cardinalities. The notion of the natural matroid of a polymatroid \cite{bonin_natural_2023} avoids this issue.
\begin{cor}
	For a polymatroid, each basis of the induced matroid is a polymatroid basis.
\end{cor}
Consequently, the polymatroid base polytope \cite{edmonds_submodular_2003,fujishige_submodular_2005,postnikov_permutohedra_2009} always contains a polytope uniquely constructed from the induced matroid base polytope: compute the lexicographically optimal base with respect to uniform weights \cite{fujishige_lexicographically_1980} (the coordinates turn out to be the critical values of the polymatroid rank from Section  \ref{Section. PP}), and use the basis change of the induced matroid to construct over vertices. In contrast, the natural matroid base polytope projects linearly onto the polymatroid base polytope. We do not develop the polytope perspective on polymatroids in this work.

We conclude the section with the study of BK-subsets of bases. Since the rank function $f$ is normalized, i.e. $f(\varnothing)=0$, we consider the empty set as a BK-set of zero cardinality with zero polymatroid rank.
\begin{lem}[\cite{schrijver_combinatorial_2003}, Theorem 44.2]	\label{Lemma. The union and intersection of BK-sets are BK-sets}
	In an independent polymatroid, unions and intersections of BK-sets are	BK-sets.
\end{lem}
\begin{proof}
	In an independent polymatroid, the defects of all subsets are non-negative. The polymatroid rank function $f$ is submodular:
	\begin{align*}
		0\leq \delta(I\cup J)+\delta(I\cap J) & \leq \delta(I)+\delta(J)\overset{I,J\text{ are BK}}{=}0.
	\end{align*}
	For non-negative defects, there is only one option: $\delta(I\cup J)=\delta(I\cap J)=0.$
\end{proof}
\begin{cor}	\label{Corollary. Distributive lattice of BK-subsets}
	For an independent polymatroid, BK-sets form a distributive lattice under inclusion.
\end{cor}
\begin{rem}
	Each basis of the matroid induced by a polymatroid is equipped with a distributive lattice of BK-sets, which forms an antimatroid \cite{edelman_meet-distributive_1980}.
\end{rem}
\begin{cor} \label{Corollary. Each basis contains a unique maximal BK-subtuple}
	For a dependent polymatroid, every basis of the induced matroid contains a unique BK-set maximal under inclusion.
\end{cor}

\section{Cycles of the induced matroid}
\subsection{Bases of cycles}
\begin{lem}	\label{Lemma. Union with a circuit decreases the defect}
	For a circuit $C$ of the induced matroid that is not contained in a subset $I$, the union $I\cup C$ has strictly smaller defect than the defect of $I$.
\end{lem}
\begin{proof}
	For such a circuit $C$, the intersection $I\cap C$ is independent
	and has non-negative defect:
	\begin{align*}
		\delta(I\cup C) =\delta(I)+\underbrace{\delta(C)}_{-1}-\underbrace{\delta(I\cap C)}_{\geq0}-\underbrace{(f(I)+f(C)-f(I\cup C)-f(I\cap C))}_{\geq0}<\delta(I).
	\end{align*}
\end{proof}
Notice that each circuit $C$ is a BK-union since $C\backslash e$ is a BK-set for each element $e\in C$.
\begin{cor} \label{Corollary. Cyclics are essential}
	Every cycle is essential and is a BK-union with negative defect.
\end{cor}
\begin{prop} \label{Proposition. Essential, rk(E)=f(E)}
	For an essential dependent polymatroid, the rank of the ground set $E$ in the induced matroid equals its polymatroid rank, $rk(E)=f(E)$.
\end{prop}
\begin{proof}
	$rk(E)=|E|+\min\,\{\delta(J),\;J\subseteq E\}\overset{\text{$E$ is essential dependent}}{=}|E|+\delta(E)=f(E)$.
\end{proof}
\begin{cor} \label{Corollary. In cyclic, bases are BK}
	In a cyclic polymatroid, bases are BK-sets.
\end{cor}
\begin{proof}
	By Corollary \ref{Corollary. Cyclics are essential}, cycles are essential. Then each basis $B$ has zero defect:
	
	$\delta(B)=f(B)-|B|\overset{\text{Theorem \ref{Theorem. f(B)=f(E)}}}{=}f(E)-|B|=f(E)-rk(E)\overset{\text{Proposition \ref{Proposition. Essential, rk(E)=f(E)}}}{=}0$.
\end{proof}
\begin{cor}
	If the bases of the induced matroid have positive defect, then the induced matroid contains a coloop.
\end{cor}
\begin{cor} \label{Corollary. BK-sets in bases have the same cardinality}
	In a dependent polymatroid, the maximal BK-sets in bases of the induced matroid have the same cardinality and contain the same set of coloops.
\end{cor}
\begin{proof}
For the maximal cycle $E'$ in the ground set $E$, the complement
$E\backslash E'$ consists of coloops. Then every basis $B$ of the
induced matroid is the union $B'\cup(E\backslash E'),$ where $B'=B\cap E'$. The set $B'$ is a basis for the restricted matroid on $E'$; hence, $B'$ is a BK-set by Corollary \ref{Corollary. In cyclic, bases are BK}. Thus the set $B'$ is contained in the maximal BK-set $I$ of the basis $B$, and $B'=I\cap E'$. By Corollary \ref{Corollary. Each basis contains a unique maximal BK-subtuple}, the set $I$ is unique, and it contains the non-empty BK-set $B'$. This means that the set difference $J=I\backslash B'$ consists of coloops, and  it is contained in every basis of the induced matroid on $E$. For the restriction of the induced matroid to the set $E'\cup J$, the BK-set $I$ is a basis, and all bases of $E'\cup J$ are BK-sets by Corollary \ref{Corollary. Basis change}.

Consider another basis $\tilde{B}$, with the basis $\tilde{B}'=\tilde{B}\cap E'$ of the maximal cycle, the maximal BK-set $\tilde{I}$ of $\tilde{B}$, and the set difference $\tilde{J}=\tilde{I}\backslash\tilde{B}'$. As above, all bases of the restriction to $E'\cup \tilde{J}$ are BK-sets. The set $B'\cup\tilde{J}$ is such a basis,
and it is contained in $B$ because $\tilde{J}$ consists of coloops; hence $B'\cup\tilde{J}\subseteq I$ by the maximality of $I$. Symmetrically, $\tilde{B}'\cup J$ is a BK-set contained in $\tilde{I}$, and the sets of coloops coincide: $J=\tilde{J}$.

Therefore, for every basis $B$ of $E$, its maximal BK-set $I$ has the same cardinality, contains the same set of coloops $J$, and equals the union $B'\cup J,$ where $B'=B\cap E'$ is a basis of the maximal cycle $E'$.
\end{proof}
\begin{rem}
	The set $J$ from the proof above is the set of coloops belonging to some BK-set of the polymatroid.
\end{rem}
\subsection{Cyclic is equivalent to essential}
In the sequel, we use the symbol $'$ to denote the maximal cycle $I'$ in a set $I$.
\begin{lem} \label{Lemma. An essential dependent polymatroid does not contain coloops}
	For	an essential dependent polymatroid, the induced matroid does not contain	coloops.
\end{lem}
\begin{proof}
	For the essential ground set $E$, the ranks satisfy: $rk(E)-rk(E')=|E\backslash E'|$.
	Since the set $E'$ is a cycle, the bases of $E'$ are BK-sets by Corollary \ref{Corollary. In cyclic, bases are BK}. We use Corollary
	\ref{Corollary. Cyclics are essential}, Proposition \ref{Proposition. Essential, rk(E)=f(E)} $rk(E')=f(E')$, and the inequality $rk(E)\leq f(E)$ to conclude $f(E)-f(E')\geq|E\backslash E'|$. This is equivalent to $\delta(E)\geq\delta(E')$, and hence 
	the essential ground set does not contain coloops.
\end{proof}
\begin{cor}	\label{Corollary. Cyclic=essential}
	For a dependent polymatroid, the following statements hold:
	\begin{enumerate}[label=\alph*\textup{)}]
		\item The polymatroid is cyclic if and only if it is essential.
		\item A subset is cyclic if and only if it is essential and has negative defect.
	\end{enumerate}
\end{cor}
\begin{cor}	\label{Corollary. cyclic, rk(I)=f(I)}
	For a dependent polymatroid, the rank of the induced matroid of a cycle $I$ equals the polymatroid rank, $rk(I)=f(I).$
\end{cor}
\begin{proof}
	Restrict the induced matroid to the set $I$ and use Proposition \ref{Proposition. Essential, rk(E)=f(E)}, Corollary \ref{Corollary. Cyclic=essential}.
\end{proof}
\begin{lem}	\label{Lemma. Unique minimal subset of the minimal defect}
	A dependent polymatroid contains a unique subset that has minimal defect and is minimal with respect to inclusion.
\end{lem}
\begin{proof}
	Suppose there exist two distinct $I$ and $J$ with this property. Then we obtain a contradiction, since the intersection $I\cap J$ has strictly greater defect and the union $I\cup J$ has strictly smaller defect than the minimum:
	\begin{align*}
		\delta(I\cup J) & =\delta(I)+\underbrace{\delta(J)-\delta(I\cap J)}_{<0}-\underbrace{(f(I)+f(J)-f(I\cup J)-f(I\cap J))}_{\geq0}<\delta(I)=\delta(J).
	\end{align*}
\end{proof}
\begin{thm} \label{Theorem. Maximal essential}
	In a dependent polymatroid, the following are equivalent:
	\begin{enumerate}[label=\alph*\textup{)}]	
		\item the subset is maximal essential and has negative defect,
		\item the subset is maximal cyclic in the induced matroid,
		\item the subset has minimal defect and is minimal with respect to inclusion.
	\end{enumerate}
\end{thm}
\begin{proof}
	By Lemmas \ref{Lemma. Union with a circuit decreases the defect}, \ref{Lemma. Unique minimal subset of the minimal defect} and Corollary \ref{Corollary. Cyclic=essential}.
\end{proof}
Some statements can be derived from the Proposition below. Our approach avoids the use of matroid duality, and results are derived from a polymatroid perspective.
\begin{prop}[\cite{shoda_large_2012}, Lemma 2.5.1]
	In a matroid, the set of cycles forms a graded lattice by the nullity function (the negative of the defect). The join of cycles is the union, and the meet is the maximal cycle in the intersection.
\end{prop}
This lattice is dual to the geometric lattice of flats of the dual matroid.

\section{Polymatroid and the induced matroid ranks}
A \textit{flat} of a polymatroid is a subset that is maximal among sets of the same rank. 

Bonin and de Mier showed that the set of cyclic flats of a matroid forms a lattice which can be arbitrary \cite{bonin_lattice_2008}. The join of cyclic flats is the closure of the union, and the meet of cyclic flats is the maximal cyclic flat inside the intersection. The lattice of cyclic flats together with a rank function yields an alternative cryptomorphic axiomatization of matroids.

The notion of polymatroid cyclic flats has been defined in \cite{freij-hollanti_matroid_2018, westerback_polymatroidal_2018, csirmaz_cyclic_2020}. In this paper, we focus on the characterization of polymatroids via their induced matroids, and consider only cyclic flats of induced matroids.
\begin{thm}[Sims \cite{sims_problems_1980}; Bonin, de Mier \cite{bonin_lattice_2008}] \label{Theorem. Sims, Bonin, Mier, rank function}
	For a subset $I$ of a matroid with the lattice of cyclic flats $\mathcal{Z}$, the matroid rank function satisfies the equality:
	\[
	rk(I)=\min \,\{rk(F)+|I\backslash F|,\;F\in\mathcal{Z}\}.
	\]
\end{thm}
\begin{prop}[\cite{shoda_large_2012}, Lemma 2.4.2]
	The rank of a matroid equals $rk(I)=rk(I')+|I\backslash I'|$.
\end{prop}
\begin{cor} \label{Corollary. New formula for rank}
	The rank of the induced matroid and the polymatroid rank are related as follows: $rk(I)=f(I')+|I\backslash I'|=|I|+\delta(I')$.
\end{cor}
Theorem \ref{Theorem. f(B)=f(E)} yields the following corollary, which will be needed in the sequel.
\begin{cor} \label{Corollary. rk(E)=f(E) iff BK-union}
	The following are equivalent:
	\begin{enumerate}[label=\alph*\textup{)}]	
		\item $f(I)=rk(I)=f(I')+|I\backslash I'|$;
		\item bases of $I$ are BK-sets;
		\item $I$ is a BK-union.
	\end{enumerate}
\end{cor}
\begin{proof}
	Choose a basis $B$ of the subset $I$, and use Theorem \ref{Theorem. f(B)=f(E)}. $(a\Rightarrow b)$ $|B|=rk(I)=f(I)=f(B)$. $(a\Leftarrow b)$ $rk(I)=|B|=f(B)=f(I)$. $(b\Rightarrow c)$ $I$ is a union of bases, which are BK-sets. 
	
	$(b\Leftarrow c)$ If a BK-union is independent, then it is a BK-set by Lemma \ref{Lemma. The union and intersection of BK-sets are BK-sets}.
	
	If a BK-union is dependent, then each basis $B$ contains a unique maximal BK-subset $J$ according to Corollary \ref{Corollary. Each basis contains a unique maximal BK-subtuple}. By Corollary \ref{Corollary. BK-sets in bases have the same cardinality}, the maximal BK-sets in bases have the same cardinality and the same set of coloops. By assumption, each coloop is contained in a BK-set, which lies in the maximal BK-set of some basis. Then the set  $B\backslash J$ consists of coloops that belong to no BK-set, and hence it is empty. Therefore, bases of a BK-union are BK-sets.
\end{proof}
\begin{rem}
	For a polymatroid, it is possible to define a lattice of BK-unions. The join of BK-unions is the union, and the meet is the maximal BK-union in the intersection. This lattice contains the lattice of cyclic subsets of the induced matroid as a sublattice.
\end{rem}

\section{Polymatroid contraction}
The \textit{contraction} $P/J$ of a polymatroid $P$ by a set
$J$ is the polymatroid $(E\backslash J,f_{J})$ with the rank function
$f_{J}(I)=f(I\cup J)-f(J)$ \cite{fujishige_polymatroidal_1978, whittle_duality_1992}. With this definition of a contraction, there are deletion-contraction formulas for the polymatroid Tutte polynomial \cite{bernardi_universal_2022,guan_deletion-contraction_2025} and the new polymatroid polynomial \cite{chavez-lomeli_new_2021}; both generalize the matroid case.

We use a subscript to denote defects after contraction: $\delta_{J}(I)=f_{J}(I)-|I|$.
\begin{lem} \label{Lemma. Contraction defect}
	For disjoint sets $I$ and $J$, we have $\delta_{J}(I)=\delta(I\cup J)-\delta(J)$.
\end{lem}
\begin{cor}
For the contraction $P/K$ of a polymatroid $P$ by a BK-set $K$,
there is a bijection between the BK-sets of the contraction $P/K$ and
the BK-sets of $P$ that contain $K$.
\end{cor}
\begin{prop} \label{Proposition. Formula for the maximal cyclic in a contraction}
	For	a set $I$ of a matroid contraction $M/K$, its maximal cycle $I_{K}'$ in $M/K$ can be computed from the maximal cycle $(I\cup K)'$ of the union $I\cup K$ in the matroid $M$:
	\[
	I'_{K}=(I\cup K)'\backslash K.
	\]
\end{prop}
\begin{proof}[First proof]
	$(\subseteq)$ By Proposition 7.4.3 of \cite{white_theory_1986}, the circuits of $M/K$ are the minimal non-empty sets of the form $C\backslash K$ for circuits $C$ of $M$. $(\supseteq)$ For a circuit $C$ of $M$ that is not contained in $K$, no element $x\in C\backslash K$ is a coloop because $rk_{M/K}(C\backslash K)=rk_{M/K}(C\backslash (K\cup x))$. Hence, the set $C\backslash K$ is a cycle in $M/K$.
\end{proof}	
\begin{proof}[Second proof]
	The standard fact about the closure operator is $cl_{M\backslash K}(I)=cl_{M}(I)\backslash K$ of \cite{oxley_matroid_2011}. This means that the closure operator commutes with deletions. By the matroid duality, the cyclic operator commutes with contractions. Denoting by $M|_{I}$ the restriction of a matroid $M$ to $I$, we obtain
	\[
	((M/K)|_{I})'=((M|_{I\cup K})/K)'=(M|_{I\cup K})'/K.
	\]
	The first equality holds because deletions and contractions commute for disjoint sets.
\end{proof}
\begin{proof}[Third proof]
	The statement is a corollary from Theorem 3.1.6 of \cite{shoda_large_2012}.
\end{proof}
\begin{lem}\label{Lemma. IUK BK-union}
	If a subset $K$ is a BK-union in a polymatroid $P$ and $I$ is a BK-union in the contraction $P/K$, then $I\cup K$ is a BK-union in the polymatroid $P$.
\end{lem}
\begin{proof}
	By Corollary \ref{Corollary. rk(E)=f(E) iff BK-union}, we can write the equalities
	\begin{align*}
		f(K)&=rk(K),\\
		f(I\cup K)-f(K)=f_K(I)&=rk_K(I)=rk(I\cup K)-rk(K),
	\end{align*}
	get $f(I\cup K)=rk(I\cup K)$, and conclude that $I\cup K$ is a BK-union in the polymatroid $P$.
\end{proof}
\begin{lem}\label{Lemma. CUK BK-union}
	Consider a BK-union $K$ in a polymatroid $P$ and a cycle $C$ in the contraction $M(P)/K$  of the induced matroid $M(P)$. Then $C\cup K$ is a BK-union of the polymatroid $P$.
\end{lem}
\begin{proof}
	Since a union of BK-unions is a BK-union, it is enough to prove the statement for a circuit $C$ of the matroid $M(P)/K$. By Proposition 3.1.10 of \cite{oxley_matroid_2011}, the circuit $C$ is the set difference $I\backslash K$ for a circuit $I$ of the matroid $M(P)$ such that $I\subseteq C\cup K$. Hence we have the equality $C\cup K=I\cup K$, and the set $C\cup K$ is a BK-union because the circuit $I$ is a BK-union by Corollary \ref{Corollary. Cyclics are essential}.
\end{proof}
\begin{thm}	\label{Theorem. Contraction by a union of BK-sets}
	For a polymatroid, the contraction by a BK-union commutes with taking the induced matroid: $M(P/K)=M(P)/K$.
\end{thm}
\begin{proof}
	Since every matroid is uniquely defined by its set of circuits, it is enough to show that the matroids $M(P/K)$ and $M(P)/K$ have the same set of circuits.
	
	Let us show that each circuit $C$ of $M(P/K)$ is a circuit of $M(P)/K$. The computation
	\begin{align*}
		|C|-1=rk_{M(P/K)}(C)&\overset{\text{Corollary \ref{Corollary. rk(E)=f(E) iff BK-union}}}{=}
		f_{P/K}(C)=f(C\cup K)-f(K)=\\
		&\overset{\text{Lemma \ref{Lemma. IUK BK-union}}}{=}rk\,(C\cup K)-rk(K)=rk_{M(P)/K}(C) 	
	\end{align*}
	shows that the ranks equal. For every element $e\in C$, the set $C\backslash e$ is a BK-set in the polymatroid $P/K$ and is a basis of the circuit $C$ of the matroid $M(P/K)$. Then the same computation provides $rk_{M(P/K)}(C\backslash e)=rk_{M(P)/K}(C\backslash e)$, and the set $C$ is a circuit of the matroid $M(P)/K$ as well.
	
	Let us show that each circuit $C$ of $M(P)/K$ is a circuit of $M(P/K)$. The computation
	\begin{align*}
		|C|-1=rk_{M(P)/K}(C)&=rk\,(C\cup K)-rk(K)=\\
		&\overset{\text{Lemma \ref{Lemma. CUK BK-union}}}{=}f(C\cup K)-f(K)=f_{P/K}(C)\geq rk_{M(P/K)}(C)
	\end{align*}
	shows that the set $C$ is dependent in the matroid $M(P/K)$. If $C$ is not a circuit of the matroid $M(P/K)$, then it contains a circuit $I$ as a proper subset. Then, we have proved, the set $I$ is a circuit of the matroid $M(P)/K$, and we obtain a contradiction with the fact that the set $C$ is a circuit of $M(P)/K$.
\end{proof}
\begin{cor}\label{Corollary. Three cycles equality}
	For	a BK-union $K$ of a polymatroid $P$ and a set $I$ of the induced matroid $M(P/K)$, the following sets coincide:
	\[
	I'_{M(P/K)}=I'_{M(P)/K}=(I\cup K)'_{M(P)}\backslash K,
	\]
	where the subscripts indicate the matroid in which the maximal cycle is computed.
\end{cor}
\begin{rem}
	Consider a BK-union $K$ of a polymatroid $P$ with the induced matroid $M(P)$.
	\begin{enumerate}[label=\textit{\alph*}\textup{)}]
		\item Then the following diagram commutes:
		$$
		\xymatrix{P\ar[r]\ar[d] & M(P)\ar[d]\\
			P/K\ar[r] & M(P/K).
		}
		$$
		\item Using Corollary \ref{Corollary. Three cycles equality}, we can verify that the rank of the induced matroid of $P/K$ equals the contracted rank of $M(P)$, i.e. $rk_{K}(I)=rk(I\cup K)-rk(K)$ for $I\subseteq E\backslash K$:
		\begin{align*}
			rk_K(I) &= f_K(I'_K)+|I\backslash I'_K|
			\tag*{Corollary \ref{Corollary. New formula for rank} for $P/K$}\\
			&= f_K\bigl((I\cup K)'\backslash K\bigr)
			+\bigl|I\backslash\bigl((I\cup K)'\backslash K\bigr)\bigr|
			\tag*{Corollary \ref{Corollary. Three cycles equality}}\\
			&= f\bigl(((I\cup K)'\backslash K)\cup K\bigr)-f(K)
			+\bigl|I\backslash\bigl((I\cup K)'\backslash K\bigr)\bigr|
			\tag*{definition of $f_K$}\\
			&= f\bigl((I\cup K)'\cup K\bigr)-f(K)+\bigl|I\backslash(I\cup K)'\bigr|
			\tag*{$I\cap K=\varnothing$}\\
			&= f\bigl((I\cup K)'\bigr)+\bigl|K\backslash(I\cup K)'\bigr|-f(K)
			+\bigl|I\backslash(I\cup K)'\bigr|
			\tag*{Corollary \ref{Corollary. rk(E)=f(E) iff BK-union} for $(I\cup K)'\cup K$}\\
			&= f\bigl((I\cup K)'\bigr)+\bigl|(I\cup K)\backslash(I\cup K)'\bigr|-f(K)
			\tag*{$I\cap K=\varnothing$}\\
			&= rk(I\cup K)-rk(K).
			\tag*{Corollary \ref{Corollary. New formula for rank}; Corollary \ref{Corollary. rk(E)=f(E) iff BK-union} for $K$}
		\end{align*}
		\item If we contract by a set $I$ that is not a BK-union, then the two ranks in the theorem above may differ. A counterexample can be found among the sets $I$ with $f(I)\neq rk(I)$. For example, for the polymatroid $([2],f)$ such that $f(\{1\})=1$, $f(\{2\})=f(\{1,2\})=2$, the contraction by the set $\{2\}$ does not commute with taking the induced matroid.
	\end{enumerate}
\end{rem}
\begin{rem}
	For independent polymatroids, contractions by BK-sets do not have loops. For dependent polymatroids, contractions by cyclic flats do not have loops.
\end{rem}
\begin{cor} \label{Corollary. Contraction by maximal cyclic is independent} 
	For the maximal cycle $E'$ of the matroid induced by a polymatroid	$P$, the contraction $P/E'$ is independent, and the induced matroid	$M(P/E')$ is free.
\end{cor}
\begin{proof}
	Since the maximal cycle $E'$ is a BK-union, the polymatroid $P/E'$ is independent because $M(P/E')=M(P)/E'$ by Theorem \ref{Theorem. Contraction by a union of BK-sets}.
\end{proof}
\begin{rem}
	The paper \cite{baker_representation_2025} provides another framework and definitions for deletions and contractions that allows one to construct a polymatroid duality similar to matroid duality. The authors also relate their definitions of polymatroid deletions and contractions to the ones used in this paper. An early version of polymatroid duality was proposed by Fujishige in 1978 \cite{fujishige_polymatroidal_1978}.
\end{rem}

\section{Principal partition}\label{Section. PP}
\subsection{Partition obtained from the Birkhoff representation theorem}
An \textit{order ideal} of a poset $\mathtt{P}$ is a subset $\mathtt{I}$ of $\mathtt{P}$ such that $\beta\in\mathtt{I}$ and $\alpha\leq\beta$ imply $\alpha\in\mathtt{I}$. For an element $\alpha$ of a poset, the \textit{principal order ideal} $(\alpha)$ is the subset of all elements $\beta$ with $\beta\leq\alpha$, and \textit{principal order filter} $[\alpha]$ is the subset of all elements $\beta$ with $\alpha\leq\beta$ \cite{stanley_enumerative_2011}.

For posets $\mathtt{P}$ and $\mathtt{Q}$, their \textit{ordinal sum} is the poset $\mathtt{P}\oplus\mathtt{Q}$ defined on the disjoint union of $\mathtt{P}$ and $\mathtt{Q}$ such that $\alpha\leq\beta$ if and only if either $\alpha\leq_{\mathtt{P}}\beta$, or $\alpha\leq_{\mathtt{Q}} \beta$, or $\alpha\in \mathtt{P}$ and $\beta\in \mathtt{Q}$.

For order lattices $\mathtt{P}$ and $\mathtt{Q}$, their \textit{glued sum} $\mathtt{P}\hat{\oplus}\mathtt{Q}$ is obtained from the ordinal sum $\mathtt{P}\oplus\mathtt{Q}$ by identification of the maximal element of the lattice $\mathtt{P}$ with the minimal element of $\mathtt{Q}$ \cite{gratzer_lattice_2011}.

In 1937, Birkhoff proved \cite{birkhoff_rings_1937} that every finite distributive lattice is isomorphic to the lattice of order ideals of some poset $\mathtt{P}$. We call such a poset the \textit{Birkhoff poset} of the distributive lattice to emphasize its origin.
\begin{lem}
	The Birkhoff poset of the glued sum $\mathcal{D}_1\hat{\oplus}\dots\hat{\oplus}\mathcal{D}_p$ of distributive lattices is the ordinal sum $\mathtt{P}_1\oplus\dots\oplus\mathtt{P}_p$ of their Birkhoff posets.
\end{lem} 
The Birkhoff poset can be used to construct partitions of finite sets  \cite{fujishige_theory_2009}.
\begin{thm}[Birkhoff]\label{Theorem. Birkhoff}
	Consider a finite set $E$ and a set of subsets $\mathcal{D}$ containing $\varnothing$ and $E$. Then $\mathcal{D}$ is a distributive lattice under union and intersection if and only if there exist a poset $\mathtt{P}$ and a partition $\Pi=\{S_\alpha,\,\alpha\in \mathtt{P}\}$ of the set $E$ such that the members $\mathcal{D}$ are exactly the sets of the form $X=\cup_{\alpha\in \mathtt{I}} \;S_{\alpha}$, where $\mathtt{I}$ ranges over the order ideals of $\mathtt{P}$.
\end{thm}
Notice that the partition $\Pi$ can be constructed by choosing an arbitrary maximal chain in the distributive lattice $\mathcal{D}$ and taking the sets of successive differences of subsets in the chain. This theorem is the crucial ingredient in the construction of principal partitions.

\subsection{Principal partition of a submodular function}
Consider a submodular function (not necessarily monotone and normalized) $s:2^E\rightarrow \mathbb{R}$ on a finite set $E$. A subset $I$ is called a \textit{minimizer} of $s$ if the submodular function $s$ attains its minimum at $I$. Applying Topkis' general result about minimizers of a submodular function on an order lattice \cite{topkis_minimizing_1978} to the particular case of the Boolean algebra $2^E$, we obtain the following well-known lemma.
\begin{lem}
	The minimizers of a submodular function on a finite set form a distributive lattice under intersection and union.
\end{lem}
Consequently, the minimizers of $s$ define a unique partition $\Pi=(E_{\min};\, S_\alpha,\, \alpha \in \mathtt{P};\, E_{\max})$ for the Birkhoff poset $\mathtt{P}$, where $E_{\min}$ consists of elements contained in every minimizer, and $E_{\max}$ consists of elements contained in no minimizer.

For a submodular function $s$, a subset $X$ is called a \textit{$\lambda$-minimizer} if the submodular function $s_\lambda(I)=s(I)-\lambda|I|$ attains its minimum at $X$. If we vary the real parameter $\lambda$, the distributive lattice of $\lambda$-minimizers $\mathcal{D}_{\lambda}$ changes. In the monograph \cite{fujishige_submodular_2005}, Fujishige provided an interpretation of this submodular function $s_\lambda$ as a Lagrangian for an optimization problem with a single constraint and the Lagrange multiplier $\lambda$. 

For a fixed $\lambda$, denote by $X^\lambda_-$ and $X^\lambda_+$ the minimal and the maximal $\lambda$-minimizers.

\begin{thm}[\cite{fujishige_submodular_2005}, Theorem 7.15] \label{Theorem. Critical values}
	For a submodular function $s$ on a finite set $E$, there exists a sequence of reals $\lambda_1<\dots<\lambda_p$ with $p\leq|E|$ such that the distributive lattice of $\lambda$-minimizers satisfies the following:
	\begin{enumerate}[label=\alph*\textup{)}]	
		\item $X^{\lambda_{i}}_-=X^{\lambda_{i-1}}_+$ with $1< i\leq p$; $X^{\lambda_{1}}_-=\varnothing$ and $X^{\lambda_p}_+=E$;
		\item for $\lambda_{i-1}<\lambda<\lambda_{i}$, $\mathcal{D}_{\lambda}=\{X^{\lambda_{i}}_-\}=\{X^{\lambda_{i-1}}_+\}$  with $1<i\leq p$; $\mathcal{D}_{\lambda}=\{\varnothing\}$ for $\lambda<\lambda_1$, and $\mathcal{D}_{\lambda}=\{E\}$ for $\lambda>\lambda_p$.
	\end{enumerate}
\end{thm}
The reals $\lambda_1,\dots,\,\lambda_p$ are called \textit{the critical values} of the submodular function $s$. 
\begin{cor}
	For a submodular function on a finite set, the following statements hold:
	\begin{enumerate}[label=\alph*\textup{)}]	
		\item The distributive lattice of $\lambda$-minimizers $\mathcal{D}_{\lambda}$ consists of a single $\lambda$-minimizer and is locally constant in $\lambda$ outside the finite set of critical values.
		\item The $\lambda$-minimizers over all $\lambda$ form a distributive lattice $\mathcal{D}$, which is the glued sum $\mathcal{D}=\mathcal{D}_{\lambda_1}\hat{\oplus}\dots\hat{\oplus}\mathcal{D}_{\lambda_p}$ for the critical values $\lambda_1,\dots,\,\lambda_p$.
		\item The Birkhoff poset $\mathtt{P}$ for the distributive lattice $\mathcal{D}$ is the ordinal sum $\mathtt{P}=\mathtt{P}_{\lambda_1}\oplus\dots\oplus\mathtt{P}_{\lambda_p}$, where $\mathtt{P}_{\lambda_i}$ is the Birkhoff poset of the distributive lattice $\mathcal{D}_{\lambda_i}$.
	\end{enumerate}
\end{cor}
For the critical values $\lambda_1<\dots<\lambda_p$, the sequence of minimal and maximal $\lambda$-minimizers $\varnothing\subset E_1\subset\dots\subset E_p=E$ in the distributive lattices $\mathcal{D}_{\lambda_i}$ is called the \textit{principal sequence}. For a critical value $\lambda_i$, the minimal and maximal $\lambda$-minimizer are $E_{i-1}$ and $E_i$, which determine the critical value to be $\lambda_i=(s(E_i)-s(E_{i-1}))/|E_i\backslash E_{i-1}|$.

We call the distributive lattice $\mathcal{D}$ \textit{the principal minimizer lattice}, and the partition obtained by the Birkhoff theorem is called the \textit{principal partition}. The principal partition has the form $\{S_\alpha,\,\alpha\in \mathtt{P}\}=\{S_\alpha^{\lambda_i},\,\alpha\in \mathtt{P}_{\lambda_i},\,1\leq i\leq p\}$. Every maximal chain in the distributive lattice $\mathcal{D}$ defines the same principal partition.
\begin{rem}
	The principal partition can be constructed by means of a certain directed graph on the ground set and the complete subgraphs of this graph \cite{fujishige_principal_1980}.
\end{rem}
\begin{rem}
	There is a related notion, the principal lattice of partitions of a submodular function, introduced by Narayanan \cite{narayanan_principal_1991}. In general, the principal lattice of partitions differs from the lattice of partitions compatible with the principal partition. However, there is a link between these lattices in graph theory. See the monograph \cite{narayanan_submodular_2009} for details.
\end{rem}
\begin{rem}
	Principal partitions have numerous applications, for example to electrical networks \cite{narayanan_submodular_2009} and to resource allocation problems \cite{fujishige_fair_1988, fujishige_theory_2009}.
\end{rem}

\subsection{Molecular and atomic polymatroids}
The notions of molecular and atomic matroids were introduced in \cite{narayanan_molecular_1981}. For polymatroids, we follow the terminology of Narayanan's monograph \cite{narayanan_submodular_2009}. We remark that the critical values of a polymatroid are always positive.

In a polymatroid, the \textit{density} of a subset $I$ is the ratio $\rho(I)=|I|/f(I)$. A polymatroid is \textit{molecular} if it has only one critical value. A polymatroid is \textit{atomic} if the principal partition is the trivial partition. Every atomic polymatroid is molecular. The ground set of an atomic polymatroid has the highest density among its subsets. Moreover, atomic polymatroids are \textit{connected}, i.e. $f(E)<f(I)+f(E\backslash I)$ for every proper subset $I\subset E$.
\begin{prop}
	A molecular polymatroid on a ground set $E$ has the critical value $1/\rho(E)$.
\end{prop}
\begin{thm}[Narayanan \cite{narayanan_submodular_2009}, Section 10.4.6]\label{Theorem. PP for restrictions and contractions}
	For a polymatroid $P$ and a $\lambda$-minimizer $X$, the following statements hold for the principal partition:
	\begin{enumerate}[label=\alph*\textup{)}]
		\item Each block contained in the $\lambda$-minimizer is a block of the principal partition of the restriction $P|_X$, with the same critical value.
		\item Each block disjoint from the $\lambda$-minimizer is a block of the principal partition of the contraction $P/X$, with the same critical value.
	\end{enumerate}
\end{thm}
\begin{cor}
	For a polymatroid $P$, its principal minimizer lattice $\mathcal{D}$, and a $\lambda$-minimizer $X$,
	\begin{enumerate}[label=\alph*\textup{)}]
		\item the principal minimizer lattice for the restriction $P|_X$ is isomorphic to the principal order ideal $(X)$ in $\mathcal{D}$;
		\item the principal minimizer lattice for the contraction $P/X$ is isomorphic to  the principal order filter $[X]$ in $\mathcal{D}$.
	\end{enumerate}
\end{cor}
\begin{thm}[Narayanan \cite{narayanan_submodular_2009}, Theorem 10.4.6]
	For a polymatroid $P$ with critical values $\lambda_1<\dots<\lambda_p$ and the principal sequence $\varnothing\subset E_1\subset\dots\subset E_p=E$, each polymatroid $(P|_{E_i})/E_{i-1}$ is molecular with the critical value $\lambda_i$.
\end{thm}
For a polymatroid $P$, denote by $\mathtt{P}$ the Birkhoff poset of the principal minimizer lattice $\mathcal{D}$. Every element $\alpha$
of the poset $\mathtt{P}$ corresponds to some block $S_{\alpha}$ with the critical value $\lambda(\alpha)$. Every order ideal $\mathtt{I}$ of the poset $\mathtt{P}$ corresponds to some $\lambda$-minimizer $S_{\mathtt{I}}=\sqcup_{\alpha\in\mathtt{I}}S_{\alpha}$.
Denote by $\hat{P}_{\alpha}$ the polymatroid minor $P_{(\alpha)}/S_{(\alpha)\backslash\alpha}$, where $(\alpha)$ is the principal order ideal and $P_{(\alpha)}$ is the restriction of the polymatroid $P$ to the $\lambda$-minimizer $S_{(\alpha)}$.
\begin{cor}
	Consider a polymatroid $P$ and its Birkhoff poset $\mathtt{P}$ for the principal minimizer lattice $\mathcal{D}$. Then the polymatroid $\hat{P}_{\alpha}$ is atomic with the critical value $\lambda(\alpha)$ and the density $1/\lambda(\alpha)$ for every element $\alpha\in\mathtt{P}$.
\end{cor}

\subsection{Critical values determine the class of a polymatroid}\label{Subsection. c.v.}
For a polymatroid, we study critical values of the defect function $\delta(I)=f(I)-|I|$ and its principal partition. In general, the defect function is not monotone. The critical values of the defect are obtained from the critical values of the polymatroid rank by subtracting one. Negative critical values of the defect correspond to the positive critical values of the polymatroid rank that are less than one. For our purposes, the defect function is the more convenient of the two, and we denote by $\delta_\lambda(I)=\delta(I)-\lambda|I|$.
\begin{prop}\label{Proposition. cv cyclic polymatroid}
	A polymatroid is cyclic if and only if all critical values of the defect are negative.
\end{prop}
\begin{proof}
	$(\Rightarrow)$ By Theorem \ref{Theorem. Maximal essential}, a cyclic polymatroid is essential with negative defect. Hence for $\lambda\geq0$, the ground set is the only $\lambda$-minimizer of the defect function:
	$$\delta_\lambda(X)-\delta_\lambda(E)=\underbrace{(\delta(X)-\delta(E))}_{\text{cyclic, }>0}+\lambda\,|E\backslash X|>0,\text{ for }X\subset E.$$
	$(\Leftarrow)$ Consider the maximal critical value $\lambda_{\max}<0$. Then for every $\lambda>\lambda_{\max}$ the ground set $E$ is the only $\lambda$-minimizer, and we have the strict inequality $\delta_\lambda(E)<\delta_\lambda(I)$ for every subset $I\subset E$. In particular, the inequality holds for $\lambda=0$, $\delta(E)<\delta(I)$, and the polymatroid is essential. For $\lambda=0$, the inequality holds for the empty set as well, $\delta(E)<0$, and the ground set is dependent and has negative defect. Hence, the polymatroid is cyclic by Theorem \ref{Theorem. Maximal essential}.
\end{proof}
\begin{rem}
	A cyclic polymatroid can have several critical values. For example, consider in the real plane two copies of one vector and three copies of another vector, where the two vectors are not parallel.
\end{rem}
\begin{prop}\label{Proposition. cv independent polymatroid}
	A polymatroid is independent if and only if all critical values of the defect are non-negative.
\end{prop}
\begin{proof}
	$(\Rightarrow)$ For $\lambda<0$, the empty set is the only $\lambda$-minimizer:
	$$\delta_\lambda(X)-\delta_\lambda(\varnothing)=\delta(X)-\lambda|X|\geq-\lambda|X|\geq0,\text{ for }X\subseteq E.$$
	$(\Leftarrow)$ For the minimal critical value of the defect function $\lambda_{\min}\geq0$ and a $\lambda_{\min}$-minimizer $X$, we have: $0=\delta_{\lambda_{\min}}(\varnothing)=\delta_{\lambda_{\min}}(X)=\delta(X)-\lambda_{\min}\,|X|$. Hence, the $\lambda_{\min}$-minimizer $X$ has non-negative defect, $\delta(X)=\lambda_{\min}\,|X|\geq0$, and every other subset in the polymatroid has non-negative defect. Therefore, the polymatroid is independent.
\end{proof}
\begin{prop}\label{Proposition. cv BK-union}
	The ground set of a polymatroid is a BK-union if and only if the critical values of the defect are non-positive.
\end{prop}
\begin{proof}
	$(\Rightarrow)$ For a BK-union $E$, we have $\delta(E)=\delta(E')$, where $E'$ is the maximal cycle. For $\lambda>0$ and a proper subset $X$, the ground set is the only $\lambda$-minimizer by Theorem \ref{Theorem. Maximal essential}:
	$$\delta_\lambda(X)-\delta_\lambda(E)=\underbrace{(\delta(X)-\delta(E'))}_{\geq 0}+\lambda\,|E\backslash X|>0,\text{ for }X\subset E.$$
	$(\Leftarrow)$ If the maximal critical value of the defect is negative, then the polymatroid is cyclic by Proposition \ref{Proposition. cv cyclic polymatroid} and is a BK-union by Corollary \ref{Corollary. Cyclics are essential}. If the maximal critical value is zero, then the ground set $E$ is a $\lambda$-minimizer of the defect at $\lambda=0$, and the inequality $\delta(E')\geq\delta(E)$ holds. 
	
	If the polymatroid is dependent, we use Theorem \ref{Theorem. Maximal essential} to get the equality $\delta(E')=\delta(E)$, which is equivalent to $f(E')+|E\backslash E'|=f(E)$. By Corollary \ref{Corollary. rk(E)=f(E) iff BK-union}, the ground set is a BK-union.
	
	If the polymatroid is independent, the maximal cycle $E'$ is empty, $0\geq\delta(E)$, and we get the equality $f(E)=|E|$. Hence, the polymatroid is a BK-polymatroid.
\end{proof}
\begin{rem}
	This Proposition provides one more characterization of BK-unions, which complements the list in Corollary \ref{Corollary. rk(E)=f(E) iff BK-union}. 
\end{rem}
\begin{lem}\label{Lemma. cv molecular}
	For a molecular polymatroid, the critical value of the defect is $\delta(E)/|E|$.
\end{lem}
\begin{prop} \label{Proposition. cv circuit}
	The ground set of a polymatroid is a circuit if and only if the polymatroid is atomic with the critical value $-1/|E|$.
\end{prop}
\begin{proof}
	$(\Rightarrow)$ For $\lambda=-1/|E|$, the empty set and the ground set are the only $\lambda$-minimizers, and $\delta(I)-\lambda|I|\geq |I|/|E|>0$  for every proper subset $I$. In particular, the empty set is the only $\lambda$-minimizer for $\lambda<-1/|E|$, and the ground set is the only $\lambda$-minimizer for $\lambda>-1/|E|$.
	
	$(\Leftarrow)$ Since the critical value is negative, the polymatroid is dependent by Proposition \ref{Proposition. cv cyclic polymatroid}. Let us show that every proper subset $I$ is independent. For the critical value $\lambda=-1/|E|$, we have the strict inequality $\delta_\lambda(I)>0$ because the polymatroid is atomic and proper subsets are not $\lambda$-minimizers. This is equivalent to $\delta(I)>-|I|/|E|$. Since the defect takes only integer values, $\delta(I)\geq0$, and the proper subset $I$ is independent. Hence, the ground set is a circuit.
\end{proof}
\begin{prop}\label{Proposition. cv essential independent}
	An essential independent polymatroid is atomic.
\end{prop}
\begin{proof}
	For $\lambda<0$, the empty set is the unique $\lambda$-minimizer by Proposition \ref{Proposition. cv independent polymatroid}. The defect of the ground set of an essential polymatroid is strictly less than the defects of its proper subsets by the definition of essential. This guarantees that the difference $\delta_\lambda(X)-\delta_\lambda(E)>0$ is strictly positive for $\lambda\geq0$ and every proper subset $X$. Then the equality can be reached only on the empty set: $\delta_\lambda(\varnothing)-\delta_\lambda(E)=-\delta(E)+\lambda|E|\geq0$ if and only if $\lambda\geq\delta(E)/|E|$. Hence the principal partition is trivial and the polymatroid is atomic.
\end{proof}
\begin{thm}\label{Theorem. cv BK-polymatroid}
	A polymatroid is a BK-polymatroid if and only if the polymatroid is molecular and the critical value of the defect function is zero.
\end{thm}
\begin{proof}
	By Proposition \ref{Proposition. cv independent polymatroid} and Proposition \ref{Proposition. cv BK-union}, all critical values are simultaneously non-negative and non-positive, hence equal to zero.
\end{proof}
\begin{rem}
	The principal partition of a BK-polymatroid is defined by its BK-sets, which are minimizers of the defect function. At the same time, BK-sets are tight sets for the all-ones vector in the polymatroid polytope.
\end{rem}
\begin{prop} \label{Proposition. cv irreducible BK-polymatroid}
	A BK-polymatroid is irreducible if and only if it is atomic.
\end{prop}
\begin{proof} 
	$(\Rightarrow)$ If the polymatroid is irreducible, then the defects of proper subsets are positive. Hence no proper subset minimizes the defect function at the critical value zero. $(\Leftarrow)$ If a BK-polymatroid is atomic, then the empty set and the ground set are the only $\lambda$-minimizers of the defect function, and the polymatroid does not contain a proper subset with zero defect. 
\end{proof}
\begin{rem}
	An notable example of an irreducible BK-polymatroid: $\dim V$ number of two-dimensional subspaces in a vector space $V$ in generic position and the dimension as the rank.
\end{rem}
\begin{cor}	\label{Corollary. BK-partition}
	For a BK-polymatroid $P$, the polymatroids $\hat{P}_{\alpha}$ are irreducible BK-polymatroids for every element $\alpha$ of the Birkhoff poset $\mathtt{P}$.
\end{cor}
This corollary is the key to decomposing discriminants of square systems of polynomial equations into irreducible strata.

\section{Realizable polymatroids}
A polymatroid $P$ is \textit{realizable} over a field if it admits a \textit{realization} by a tuple of vector subspaces $(L_{1},\dots,L_{n})$ in a vector space $V$ with rank function $f(I)=\dim L_{I}$, $L_{I}=\sum_{i\in I}L_{i}$. 

Rado's theorem on independent transversals provides another characterization of independent subsets in realizable polymatroids via their realizations.
\begin{thm}[Rado \cite{rado_theorem_1942}]
	A tuple of vector subspaces $(L_{1},\dots,L_{n})$ corresponds to an independent realizable polymatroid if and only if there exists a linearly independent set of vectors $v_{1}\in L_{1}$, \dots, $v_{n}\in L_{n}$.
\end{thm}
Khovanskii gave another proof specifically for vector subspaces in Theorem 4 of \cite{khovanskii_newton_2016}.

\subsection{Partition of a vector space by a realizable polymatroid}
A \textit{dual realization} of a polymatroid $P$ is a tuple
of subspaces $(L_{1},\dots,L_{n})$ in the dual vector space $V^\vee$ with the rank function $f(I)=\codim L_{I}$, $L_{I}=\cap_{i\in I}L_{i}$. Both realizations determine the same polymatroid.

In this subsection, we consider polymatroids with a lattice of flats
$\mathcal{L}$ and their dual realizations.
\begin{lem}\label{Lemma. For a point x, I is flat}
	For a point $x\in V^\vee$, the set $I=\{i\in[n]\,|\;L_{i}\ni x\}$ is a flat.
\end{lem}
\begin{proof}
	If the set $I$ is not a flat, there exists an element $j\in F\backslash I$, where $F$ is the polymatroid closure of the subset $I$. Then the point $x$ belongs to the subspace $L_I=L_F$ but not to the subspace $L_{j}$. We obtain a contradiction since $L_{F}\subseteq L_{j}$.
\end{proof}
\begin{lem}\label{Lemma. B_F}
	In the dual space $V^\vee$, the set of points corresponding to a flat $F$ is
	\[
	B_{F}=L_{F}\backslash\underset{F'\in\mathcal{L}\backslash(F)}{\cup}L_{F'},
	\]
	where $(F)$ is the principal order ideal in the lattice of flats
	$\mathcal{L}$.
\end{lem}
\begin{proof}
	The points corresponding to a flat $F$ belong to the subspace $L_{F}$. However, we need to remove points that do not correspond to the flat	$F$. An inclusion of flats $F'\subset F$ yields the reverse inclusion of subspaces $L_{F'}\supset L_{F}$. Hence the subspace $L_{F}$ lies in the subspace $L_{F'}$ only for flats $F'$ of the principal order ideal $(F)$ in the lattice $\mathcal{L}$, and it does not lie in the subspaces of other flats. This determines the	set $B_{F}=L_{F}\backslash\cup_{F'\in\mathcal{L}\backslash(F)}L_{F'}.$
\end{proof}
\begin{prop}\label{Proposition. Polymatroid partition of a vector space}
	Consider a realizable polymatroid with the lattice of flats $\mathcal{L}$ and a dual realization by a tuple of subspaces $(L_{1},\dots,L_{n})$ in the dual vector space $V^\vee$ over an infinite field. Then there exists a partition of the dual space $V^\vee=\sqcup_{F\in\mathcal{L}}B_{F}$ into constructible subsets $B_{F}$ having the closure $L_F$.
\end{prop}
\begin{proof}
	Since every point of the vector space $V^\vee$ corresponds to a unique flat, Lemma \ref{Lemma. For a point x, I is flat} provides a set-theoretic map $\gamma:V^\vee\rightarrow\mathcal{L}$. This map defines a partition of the space into subsets $\gamma^{-1}(F)=B_{F}$ by Lemma \ref{Lemma. B_F}. Over a finite field, the set $B_{F}$ may be empty. Over an infinite field, $B_{F}$ is nonempty dense subset of $L_F$ as a finite intersection of dense subsets $B_{F}=\cap_{F'\in\mathcal{L}\backslash(F)}(L_{F}\backslash L_{F'}).$
\end{proof}

\subsection{Quotients and realizable contractions}
In this subsection, we use a realization $\mathsf{n}=(L_{1},\dots,L_{n})$
of a polymatroid with the rank function $f(I)=\dim \langle\mathsf{n}_{I}\rangle$, where $\mathsf{n}_{I}=(L_{i},\;i\in I)$ denotes the subtuple and $\langle\mathsf{n}_{I}\rangle=\sum_{i\in I}L_{i}$ its sum.

For subtuples $\mathsf{n}_{I}$ and $\mathsf{n}_{J}$ of a tuple
of vector subspaces $\mathsf{n}$ in a vector space $V$, the \textit{quotient tuple} is defined as $\mathsf{n}_{J}/\mathsf{n}_{I}=\pi(\mathsf{n}_{J\backslash I})$, where $\pi:V\rightarrow V/\langle\mathsf{n}_{I}\rangle$ is the projection.
\begin{prop} \label{Proposition. Polymatroid contraction and quotient}
	Let a tuple	of subspaces $\mathsf{n}$ be a realization of a polymatroid $P$. Then for any subsets $I$ and $J$, the restriction of the polymatroid contraction $P/I$ to  $J\backslash I$ is realized by the quotient tuple $\mathsf{n}_{J}/\mathsf{n}_{I}$.
\end{prop}
\begin{proof}
	The polymatroid rank function for the contraction $P/I$ is equal
	to the dimension of the quotient tuple of the realization $\mathsf{n}$:
	\begin{align*}
		f_{I}(J) & =f(I\cup J)-f(I)=\dim\,\langle\mathsf{n}_{I\cup J}\rangle-\dim\,\langle\mathsf{n}_{I}\rangle =\dim\,\langle\mathsf{n}_{I\cup J}\rangle/\langle\mathsf{n}_{I}\rangle\\
		& =\dim\,\langle\mathsf{n}_{J\backslash I}\rangle/\langle\mathsf{n}_{I}\rangle =\dim\,\pi(\langle\mathsf{n}_{J\backslash I}\rangle) =\dim\,\langle\pi(\mathsf{n}_{J\backslash I})\rangle =\dim\,\langle\mathsf{n}_{J}/\mathsf{n}_{I}\rangle.
	\end{align*}
\end{proof}
\begin{rem}
	A realization $\mathsf{n}$ of a polymatroid $P$, a polymatroid contraction $P/I$, and the quotient $\mathsf{n}/\mathsf{n}_{I}$ are compatible in the following commutative diagram:
	$$\xymatrix{P\ar[rr]^{\text{realization}}\ar[d]_{\text{contraction}} &  & \mathsf{n}\ar[d]^{\text{quotient}}\\
		P/I\ar[rr]^{\text{realization}} &  & \mathsf{n}/\mathsf{n}_{I}.}
	$$
\end{rem}
\begin{rem}
	We introduce the separate notion of quotient tuples of vector subspaces to emphasize that polymatroid contractions of realizable polymatroids can be expressed as quotients in linear algebra, in contrast to the abstract case.
\end{rem}
\begin{rem}
	Consider a polymatroid realization via a tuple of subspaces $(L_{1},\dots,L_{n})$ and a vector subspace $U$. We can define the polymatroid quotient as the tuple $(L_{1}/U,\dots,L_{n}/U)$. Then for the induced matroids we obtain the matroid quotient in the sense of \cite{white_theory_1986}.
\end{rem}

\subsection{Bicolored posets for strata of discriminants}
A poset is \textit{bicolored} if it is endowed with a set-theoretic partition into two blocks called \textit{colors}. Bicolored posets have bivariate polynomial invariants \cite{beck_bivariate_2020}.

For discriminants, we call one color prelinear and the other nonlinear. For a non-degenerate polynomial system with an equal number of equations and variables (a square system), strata of the discriminant carry a partial order. Prelinear strata have codimension two, and nonlinear strata have codimension one.
\begin{prop}\label{Proposition. Bicoloring}
	Any bicolored poset can index the strata of a square polynomial system.
\end{prop}
\begin{proof}
	Consider a poset $\mathtt{P}$ and a map $m:\mathtt{P}\rightarrow\{0,1\}$, where $0$ corresponds to the prelinear color and $1$ corresponds to the nonlinear color. Introduce variables $x_\alpha$ for every element $\alpha$ of the bicolored poset $\mathtt{P}$. Then the square polynomial system with indeterminate coefficients $a_\alpha$, $b_{\alpha\beta}$, $c_\alpha$, whose discriminant strata have the prescribed codimensions, can have the form:
	$$a_{\alpha}m(\alpha)\,x_{\alpha}^{2}+\underset{\beta\leq\alpha}{\sum}b_{\alpha\beta}x_{\beta}+c_{\alpha}=0,\qquad\alpha\in\mathtt{P}.$$
	For $m(\alpha)=0$, the $\alpha$-th equation is linear; for $m(\alpha)=1$, the $\alpha$-th equation is quadratic in $x_\alpha$. In that case, the principal partition is the singleton partition, and every element of $\mathtt{P}$ corresponds to a stratum. See \cite{pokidkin_components_2025} for a verification that the strata of the discriminant have the prescribed codimensions. See Theorem \ref{Theorem. Components for the discriminant with BK-polymatroid} to determine which strata are components.
\end{proof}

\section*{Acknowledgments}
I am grateful to Alexander Esterov and my family.

\section*{AI disclosure}
All results were obtained by the author; no proofs were generated by artificial intelligence. Publicly available AI models (DeepSeek V4, Perplexity, and Gemini 3.5 Flash, accessed June 2026) suggested the initial, informal link between essential polymatroids and principal partitions and assisted with the literature search; the author then made this connection precise and independently obtained the analogous description for all other polymatroid classes considered here. Claude Fable 5 (Claude Max, July 2026) assisted with editing — detecting logical flaws, improving the LaTeX code, and enhancing clarity.

\printbibliography

@article{postnikov_permutohedra_2009,
	title = {Permutohedra, {Associahedra}, and {Beyond}},
	volume = {2009},
	doi = {10.1093/imrn/rnn153},
	number = {6},
	journal = {International Mathematics Research Notices},
	author = {Postnikov, Alexander},
	year = {2009},
	pages = {1026--1106},
	note = {\href{https://arxiv.org/abs/math/0507163}{ArXiv}},
}

@article{bernardi_universal_2022,
	title = {Universal {Tutte} polynomial},
	volume = {402},
	doi = {10.1016/j.aim.2022.108355},
	journal = {Advances in Mathematics},
	author = {Bernardi, Olivier and Kálmán, Tamás and Postnikov, Alexander},
	month = jun,
	year = {2022},
	pages = {108355},
	note = {\href{https://arxiv.org/abs/2004.00683}{ArXiv}},
}

@phdthesis{sims_problems_1980,
	type = {Ph.{D}. {Dissertation}},
	title = {Some {Problems} in {Matroid} {Theory}},
	school = {Linacre College, University of Oxford},
	author = {Sims, Julie A.},
	year = {1980},
}

@incollection{edmonds_submodular_2003,
	address = {Berlin, Heidelberg},
	series = {Lecture {Notes} in {Computer} {Science}},
	title = {Submodular {Functions}, {Matroids}, and {Certain} {Polyhedra}},
	volume = {2570},
	doi = {10.1007/3-540-36478-1_2},
	booktitle = {Combinatorial {Optimization} — {Eureka}, {You} {Shrink}!},
	publisher = {Springer Berlin Heidelberg},
	author = {Edmonds, Jack},
	editor = {Goos, Gerhard and Hartmanis, Juris and Van Leeuwen, Jan and Jünger, Michael and Reinelt, Gerhard and Rinaldi, Giovanni},
	year = {2003},
	pages = {11--26},
}

@incollection{fujishige_theory_2009,
	address = {Berlin, Heidelberg},
	title = {Theory of {Principal} {Partitions} {Revisited}},
	doi = {10.1007/978-3-540-76796-1_7},
	booktitle = {Research {Trends} in {Combinatorial} {Optimization}},
	publisher = {Springer Berlin Heidelberg},
	author = {Fujishige, Satoru},
	editor = {Cook, William and Lovász, László and Vygen, Jens},
	year = {2009},
	pages = {127--162},
}

@incollection{narayanan_molecular_1981,
	address = {Berlin, Heidelberg},
	series = {Lecture {Notes} in {Mathematics}},
	title = {On molecular and atomic matroids},
	volume = {885},
	doi = {10.1007/BFb0092281},
	booktitle = {Combinatorics and {Graph} {Theory}},
	publisher = {Springer Berlin Heidelberg},
	author = {Narayanan, Harihar and Vartak, Manohar N.},
	editor = {Rao, Siddani Bhaskara},
	year = {1981},
	pages = {358--364},
}

@book{gratzer_lattice_2011,
	address = {Basel},
	title = {Lattice {Theory}: {Foundation}},
	doi = {10.1007/978-3-0348-0018-1},
	publisher = {Springer Basel},
	author = {Grätzer, George},
	year = {2011},
}

@article{brysiewicz_solving_2021,
	title = {Solving decomposable sparse systems},
	volume = {88},
	doi = {10.1007/s11075-020-01045-x},
	number = {1},
	journal = {Numerical Algorithms},
	author = {Brysiewicz, Taylor and Rodriguez, Jose Israel and Sottile, Frank and Yahl, Thomas},
	month = sep,
	year = {2021},
	note = {\href{https://arxiv.org/abs/2001.04228v2}{ArXiv}},
	pages = {453--474},
}

@article{guan_deletion-contraction_2025,
	title = {A {Deletion}-{Contraction} {Formula} and {Monotonicity} {Properties} for the {Polymatroid} {Tutte} {Polynomial}},
	volume = {2025},
	doi = {10.1093/imrn/rnaf302},
	number = {19},
	journal = {International Mathematics Research Notices},
	author = {Guan, Xiaxia and Jin, Xian’an and Kálmán, Tamás},
	month = oct,
	year = {2025},
	note = {\href{https://arxiv.org/abs/2309.13639}{ArXiv}},
	pages = {rnaf302},
}

@incollection{freij-hollanti_matroid_2018,
	address = {Cham},
	title = {Matroid {Theory} and {Storage} {Codes}: {Bounds} and {Constructions}},
	doi = {10.1007/978-3-319-70293-3_15},
	booktitle = {Network {Coding} and {Subspace} {Designs}},
	publisher = {Springer International Publishing},
	author = {Freij-Hollanti, Ragnar and Hollanti, Camilla and Westerbäck, Thomas},
	editor = {Greferath, Marcus and Pavčević, Mario Osvin and Silberstein, Natalia and Vázquez-Castro, María Ángeles},
	year = {2018},
	note = {\href{https://arxiv.org/abs/1704.04007v2}{ArXiv}},
	pages = {385--425},
}

@article{csirmaz_cyclic_2020,
	title = {Cyclic {Flats} of a {Polymatroid}},
	volume = {24},
	doi = {10.1007/s00026-020-00506-3},
	number = {4},
	journal = {Annals of Combinatorics},
	author = {Csirmaz, László},
	month = dec,
	year = {2020},
	note = {\href{https://arxiv.org/abs/1910.05756}{ArXiv}},
	pages = {637--648},
}

@article{crowley_bergman_2026,
	title = {The {Bergman} fan of a polymatroid},
	volume = {379},
	doi = {10.1090/tran/9492},
	number = {5},
	journal = {Transactions of the American Mathematical Society},
	author = {Crowley, Colin and Huh, June and Larson, Matt and Simpson, Connor and Wang, Botong},
	month = feb,
	year = {2026},
	note = {\href{https://arxiv.org/abs/2207.08764v2}{ArXiv}},
	pages = {3239--3256},
}

@article{cattani_non-splitting_2022,
	title = {Non-splitting {Flags}, {Iterated} {Circuits}, $\underline{\sigma}$-{Matrices} and {Cayley} {Configurations}},
	volume = {50},
	doi = {10.1007/s10013-022-00554-7},
	number = {3},
	journal = {Vietnam Journal of Mathematics},
	author = {Cattani, Eduardo and Dickenstein, Alicia},
	month = jul,
	year = {2022},
	note = {\href{https://arxiv.org/abs/2105.00302}{ArXiv}},
	pages = {679--706},
}

@article{bonin_natural_2023,
	title = {The {Natural} {Matroid} of an {Integer} {Polymatroid}},
	volume = {37},
	doi = {10.1137/22M1521122},
	number = {3},
	journal = {SIAM Journal on Discrete Mathematics},
	author = {Bonin, Joseph E. and Chun, Carolyn and Fife, Tara},
	month = sep,
	year = {2023},
	note = {\href{https://arxiv.org/abs/2209.03786}{ArXiv}},
	pages = {1751--1770},
}

@article{beck_bivariate_2020,
	title = {Bivariate {Order} {Polynomials}},
	volume = {36},
	doi = {10.1007/s00373-019-02128-w},
	number = {4},
	journal = {Graphs and Combinatorics},
	author = {Beck, Matthias and Farahmand, Maryam and Karunaratne, Gina and Zuniga Ruiz, Sandra},
	month = jul,
	year = {2020},
	note = {\href{https://arxiv.org/abs/1901.06720}{ArXiv}},
	pages = {921--931},
}

@article{rado_theorem_1942,
	title = {A theorem on independence relations},
	volume = {os-13},
	doi = {10.1093/qmath/os-13.1.83},
	number = {1},
	journal = {The Quarterly Journal of Mathematics},
	author = {Rado, Richard},
	year = {1942},
	pages = {83--89},
}

@book{narayanan_submodular_2009,
	address = {Mumbai},
	edition = {Revised Open Edition},
	title = {Submodular {Functions} and {Electrical} {Networks}},
	publisher = {Department of Electrical Engineering, Indian Institute of Technology Bombay},
	author = {Narayanan, Harihar},
	month = feb,
	year = {2009},
	note = {\href{https://www.ee.iitb.ac.in/~hn/book/SubmodularFunction_2ed.pdf}{PDF}},
}

@book{gelfand_discriminants_2008,
	address = {Boston},
	series = {Modern {Birkhäuser} {Classics}},
	title = {Discriminants, {Resultants}, and {Multidimensional} {Determinants}},
	publisher = {Birkhäuser Boston},
	author = {Gelfand, Israel M. and Kapranov, Mikhail M. and Zelevinsky, Andrei V.},
	year = {2008},
	doi = {10.1007/978-0-8176-4771-1},
}

@book{stanley_enumerative_2011,
	edition = {2},
	series = {Cambridge {Studies} in {Advanced} {Mathematics}},
	title = {Enumerative {Combinatorics}: {Volume} 1},
	volume = {49},
	publisher = {Cambridge University Press},
	author = {Stanley, Richard P.},
	year = {2011},
}

@book{schrijver_combinatorial_2003,
	series = {Algorithms and {Combinatorics}},
	title = {Combinatorial {Optimization}: {Polyhedra} and {Efficiency}},
	volume = {24},
	doi = {10.1007/978-3-540-44389-6},
	publisher = {Springer-Verlag Berlin Heidelberg},
	author = {Schrijver, Alexander},
	year = {2003},
}

@article{pokidkin_components_2025,
	title = {Components of discriminants for systems of equations and irreducibility of determinants},
	volume = {2026},
	doi = {10.1093/imrn/rnag213},
	number = {19},
	journal = {International Mathematics Research Notices},
	author = {Pokidkin, Vladislav},
	year = {2026},
	note = {\href{https://arxiv.org/abs/2501.15832v3}{ArXiv}},
	pages = {rnag213},
}

@article{esterov_discriminant_2013,
	title = {The discriminant of a system of equations},
	volume = {245},
	doi = {10.1016/j.aim.2013.06.027},
	journal = {Advances in Mathematics},
	author = {Esterov, Alexander},
	month = oct,
	year = {2013},
	note = {\href{https://arxiv.org/abs/1110.4060v2}{ArXiv}},
	pages = {534--572},
}

@book{white_theory_1986,
	series = {Encyclopedia of {Mathematics} and its {Applications}},
	title = {Theory of {Matroids}},
	volume = {26},
	doi = {10.1017/CBO9780511629563},
	publisher = {Cambridge University Press},
	editor = {White, Neil},
	year = {1986},
}

@book{welsh_matroid_1976,
	address = {London, New York},
	series = {L.{M}.{S}. {Monographs}},
	title = {Matroid theory},
	volume = {8},
	publisher = {Academic Press},
	author = {Welsh, Dominic James Anthony},
	year = {1976},
	note = {Open Library ID: OL4881814M},
}

@article{khovanskii_newton_1978,
	title = {{Newton} polyhedra and the genus of complete intersections},
	volume = {12},
	doi = {10.1007/BF01077562},
	number = {1},
	journal = {Functional Analysis and Its Applications},
	author = {Khovanskii, Askold G.},
	month = jan,
	year = {1978},
	pages = {38--46},
}

@misc{baker_representation_2025,
	title = {Representation theory for polymatroids},
	doi = {10.48550/arXiv.2507.14718},
	publisher = {arXiv},
	author = {Baker, Matthew and Huh, June and Kim, Donggyu and Kummer, Mario and Lorscheid, Oliver},
	month = sep,
	year = {2025},
	note = {\href{https://arxiv.org/abs/2507.14718v2}{ArXiv}},
}

@article{kishi_maximally_1968,
	title = {Maximally distant trees in a linear graph},
	volume = {51-A},
	number = {5},
	journal = {Transactions of the Institute of Electronics and Communication Engineers of Japan, Section A},
	author = {Kishi, Genya and Kajitani, Yutaka},
	year = {1968},
	pages = {196--203},
	note = {(in {Japanese})},
}

@article{chavez-lomeli_new_2021,
	title = {A new polynomial for polymatroids},
	volume = {80},
	number = {3},
	journal = {Australasian Journal of Combinatorics},
	author = {Chávez-Lomelí, Laura and Merino, Criel and Rodríguez Sánchez, Guadalupe and Whittle, Geoffrey},
	year = {2021},
	pages = {342--360},
}

@article{fujishige_principal_1980,
	title = {Principal structures of submodular systems},
	volume = {2},
	doi = {10.1016/0166-218X(80)90057-8},
	number = {1},
	journal = {Discrete Applied Mathematics},
	author = {Fujishige, Satoru},
	month = apr,
	year = {1980},
	pages = {77--79},
}

@article{fujishige_fair_1988,
	title = {The {Fair} {Resource} {Allocation} {Problem} with {Submodular} {Constraints}},
	volume = {13},
	doi = {10.1287/moor.13.1.164},
	number = {1},
	journal = {Mathematics of Operations Research},
	author = {Fujishige, Satoru and Katoh, Naoki and Ichimori, Tetsuo},
	month = feb,
	year = {1988},
	pages = {164--173},
}

@article{narayanan_principal_1991,
	title = {The principal lattice of partitions of a submodular function},
	volume = {144},
	doi = {10.1016/0024-3795(91)90070-D},
	journal = {Linear Algebra and its Applications},
	author = {Narayanan, Harihar},
	month = jan,
	year = {1991},
	pages = {179--216},
}

@article{topkis_minimizing_1978,
	title = {Minimizing a {Submodular} {Function} on a {Lattice}},
	volume = {26},
	doi = {10.1287/opre.26.2.305},
	number = {2},
	journal = {Operations Research},
	author = {Topkis, Donald M.},
	month = apr,
	year = {1978},
	pages = {305--321},
}

@article{fujishige_polymatroidal_1978,
	title = {Polymatroidal dependence structure of a set of random variables},
	volume = {39},
	doi = {10.1016/S0019-9958(78)91063-X},
	number = {1},
	journal = {Information and Control},
	author = {Fujishige, Satoru},
	month = oct,
	year = {1978},
	pages = {55--72},
}

@article{fujishige_lexicographically_1980,
	title = {Lexicographically {Optimal} {Base} of a {Polymatroid} with {Respect} to a {Weight} {Vector}},
	volume = {5},
	doi = {10.1287/moor.5.2.186},
	number = {2},
	journal = {Mathematics of Operations Research},
	author = {Fujishige, Satoru},
	month = may,
	year = {1980},
	pages = {186--196},
}

@book{fujishige_submodular_2005,
	edition = {2},
	series = {Annals of {Discrete} {Mathematics}},
	title = {Submodular {Functions} and {Optimization}},
	volume = {58},
	publisher = {Elsevier},
	author = {Fujishige, Satoru},
	year = {2005},
	doi = {10.1016/S0167-5060(05)X8001-3},
}

@inproceedings{westerback_polymatroidal_2018,
	title = {On the {Polymatroidal} {Structure} of {Quasi}-{Uniform} {Codes} with {Applications} to {Heterogeneous} {Distributed} {Storage}},
	booktitle = {Proceedings of the 23rd {International} {Symposium} on {Mathematical} {Theory} of {Networks} and {Systems} ({MTNS} 2018)},
	address = {Hong Kong},
	publisher = {Hong Kong University of Science and Technology},
	author = {Westerbäck, Thomas and Grezet, Matthias and Freij-Hollanti, Ragnar and Hollanti, Camilla},
	month = jul,
	year = {2018},
	pages = {641--647},
}

@article{whittle_duality_1992,
	title = {Duality in {Polymatroids} and {Set} {Functions}},
	volume = {1},
	doi = {10.1017/S0963548300000304},
	number = {3},
	journal = {Combinatorics, Probability and Computing},
	author = {Whittle, Geoff},
	month = sep,
	year = {1992},
	pages = {275--280},
}

@phdthesis{shoda_large_2012,
	type = {Ph.{D}. {Dissertation}},
	title = {Large {Families} of {Matroids} with the {Same} {Tutte} {Polynomial}},
	school = {The George Washington University},
	author = {Shoda, Ken},
	month = aug,
	year = {2012},
}

@article{mcdiarmid_independence_1973,
	title = {Independence {Structures} and {Submodular} {Functions}},
	volume = {5},
	doi = {10.1112/blms/5.1.18},
	number = {1},
	journal = {Bulletin of the London Mathematical Society},
	author = {McDiarmid, Colin},
	month = mar,
	year = {1973},
	pages = {18--20},
}

@misc{swartz_polymatroids_2024,
	title = {Polymatroids are to finite groups as matroids are to finite fields},
	doi = {10.48550/ARXIV.2402.17582},
	publisher = {arXiv},
	author = {Swartz, Ed and Wentworth-Nice, Prairie and Xue, Alexander},
	year = {2024},
	note = {\href{https://arxiv.org/abs/2402.17582}{ArXiv}},
}

@article{furukawa_combinatorial_2021,
	title = {A combinatorial description of dual defects of toric varieties},
	volume = {23},
	doi = {10.1142/S0219199720500017},
	number = {01},
	journal = {Communications in Contemporary Mathematics},
	author = {Furukawa, Katsuhisa and Ito, Atsushi},
	month = feb,
	year = {2021},
	note = {\href{https://arxiv.org/abs/1605.05801v2}{ArXiv}},
	pages = {2050001},
}

@article{dickenstein_iterated_2023,
	title = {Iterated and mixed discriminants},
	volume = {7},
	doi = {10.4171/jca/68},
	number = {1},
	journal = {Journal of Combinatorial Algebra},
	author = {Dickenstein, Alicia and Di Rocco, Sandra and Morrison, Ralph},
	month = may,
	year = {2023},
	note = {\href{https://arxiv.org/abs/2101.11571v2}{ArXiv}},
	pages = {45--81},
}

@article{kaveh_mixed_2010,
	title = {Mixed {Volume} and an {Extension} of {Intersection} {Theory} of {Divisors}},
	volume = {10},
	doi = {10.17323/1609-4514-2010-10-2-343-375},
	number = {2},
	journal = {Moscow Mathematical Journal},
	author = {Kaveh, Kiumars and Khovanskii, Askold G.},
	month = aug,
	year = {2010},
	note = {\href{https://arxiv.org/abs/0812.0433v2}{ArXiv}},
	pages = {343--375},
}

@article{jensen_computing_2013,
	title = {Computing tropical resultants},
	volume = {387},
	doi = {10.1016/j.jalgebra.2013.03.031},
	journal = {Journal of Algebra},
	author = {Jensen, Anders and Yu, Josephine},
	month = aug,
	year = {2013},
	note = {\href{https://arxiv.org/abs/1109.2368v2}{ArXiv}},
	pages = {287--319},
}

@article{esterov_galois_2019,
	title = {{Galois} theory for general systems of polynomial equations},
	volume = {155},
	doi = {10.1112/S0010437X18007868},
	number = {2},
	journal = {Compositio Mathematica},
	author = {Esterov, Alexander},
	month = feb,
	year = {2019},
	note = {\href{https://arxiv.org/abs/1801.08260v3}{ArXiv}},
	pages = {229--245},
}

@article{esterov_newton_2010,
	title = {Newton {Polyhedra} of {Discriminants} of {Projections}},
	volume = {44},
	doi = {10.1007/s00454-010-9242-7},
	number = {1},
	journal = {Discrete \& Computational Geometry},
	author = {Esterov, Alexander},
	month = jul,
	year = {2010},
	note = {\href{https://arxiv.org/abs/0810.4996v3}{ArXiv}},
	pages = {96--148},
}

@article{esterov_multivariate_2016,
	title = {Multivariate {Abel}–{Ruffini}},
	volume = {365},
	doi = {10.1007/s00208-015-1309-6},
	number = {3-4},
	journal = {Mathematische Annalen},
	author = {Esterov, Alexander and Gusev, Gleb},
	month = aug,
	year = {2016},
	note = {\href{https://arxiv.org/abs/1405.1252v3}{ArXiv}},
	pages = {1091--1110},
}

@article{cattani_mixed_2013,
	title = {Mixed discriminants},
	volume = {274},
	doi = {10.1007/s00209-012-1095-8},
	number = {3},
	journal = {Mathematische Zeitschrift},
	author = {Cattani, Eduardo and Cueto, María Angélica and Dickenstein, Alicia and Di Rocco, Sandra and Sturmfels, Bernd},
	month = aug,
	year = {2013},
	note = {\href{https://arxiv.org/abs/1112.1012v1}{ArXiv}},
	pages = {761--778},
}

@article{bonin_lattice_2008,
	title = {The {Lattice} of {Cyclic} {Flats} of a {Matroid}},
	volume = {12},
	doi = {10.1007/s00026-008-0344-3},
	number = {2},
	journal = {Annals of Combinatorics},
	author = {Bonin, Joseph E. and de Mier, Anna},
	month = jul,
	year = {2008},
	note = {\href{https://arxiv.org/abs/math/0505689v2}{ArXiv}},
	pages = {155--170},
}

@book{oxley_matroid_2011,
	edition = {2},
	series = {Oxford {Graduate} {Texts} in {Mathematics}},
	title = {Matroid {Theory}},
	volume = {21},
	publisher = {Oxford University Press},
	author = {Oxley, James G.},
	month = feb,
	year = {2011},
	doi = {10.1093/acprof:oso/9780198566946.001.0001},
}

@article{sturmfels_newton_1994,
	title = {On the {Newton} {Polytope} of the {Resultant}},
	volume = {3},
	doi = {10.1023/A:1022497624378},
	number = {2},
	journal = {Journal of Algebraic Combinatorics},
	author = {Sturmfels, Bernd},
	month = apr,
	year = {1994},
	pages = {207--236},
}

@article{birkhoff_rings_1937,
	title = {Rings of sets},
	volume = {3},
	doi = {10.1215/S0012-7094-37-00334-X},
	number = {3},
	journal = {Duke Mathematical Journal},
	author = {Birkhoff, Garrett},
	month = sep,
	year = {1937},
	pages = {443--454},
}

@article{bernshtein_number_1975,
	title = {The number of roots of a system of equations},
	volume = {9},
	doi = {10.1007/BF01075595},
	number = {3},
	journal = {Functional Analysis and Its Applications},
	author = {Bernshtein, David N.},
	month = jul,
	year = {1975},
	pages = {183--185},
}

@article{khovanskii_newton_2016,
	title = {{Newton} polytopes and irreducible components of complete intersections},
	volume = {80},
	doi = {10.1070/IM8307},
	number = {1},
	journal = {Izvestiya: Mathematics},
	author = {Khovanskii, Askold G.},
	month = feb,
	year = {2016},
	note = {Mathnet.ru/eng/im8307},
	pages = {263--284},
}

@article{eur_intersection_2024,
	title = {Intersection {Theory} of {Polymatroids}},
	volume = {2024},
	doi = {10.1093/imrn/rnad213},
	number = {5},
	journal = {International Mathematics Research Notices},
	author = {Eur, Christopher and Larson, Matt},
	month = mar,
	year = {2024},
	note = {\href{https://arxiv.org/abs/2301.00831}{ArXiv}},
	pages = {4207--4241},
}

@article{pagaria_hodge_2023,
	title = {Hodge {Theory} for {Polymatroids}},
	volume = {2023},
	doi = {10.1093/imrn/rnad001},
	number = {23},
	journal = {International Mathematics Research Notices},
	author = {Pagaria, Roberto and Pezzoli, Gian Marco},
	month = dec,
	year = {2023},
	note = {\href{https://arxiv.org/abs/2105.04214v2}{ArXiv}},
	pages = {20118--20168},
}

@article{edelman_meet-distributive_1980,
	title = {Meet-distributive lattices and the anti-exchange closure},
	volume = {10},
	doi = {10.1007/BF02482912},
	number = {1},
	journal = {Algebra Universalis},
	author = {Edelman, Paul H.},
	month = dec,
	year = {1980},
	pages = {290--299},
}

\section*{Appendix}

Each picture corresponds to a configuration of coordinate subspaces. The rows correspond to the vectors of the standard basis of the vector space. The coloured squares in a column correspond to the nonzero generators that span the subspace. The number of subspaces equals the number of columns of the picture, and the dimension of the sum of all subspaces of the configuration equals the number of rows. The critical values of the defect function are given below each picture. The dotted lines also show the partition into blocks.

We list all pictures with at most six coloured squares, up to permutations of rows and columns. The pictures are AI-generated.

\section*{A. Dependent polymatroids}
At least one critical value is negative.

\subsection*{A1. Cyclic polymatroids $=$ essential dependent polymatroids}
All critical values are negative.

\subsubsection*{A1a. Circuits}
Atomic, with the single critical value $-1/|E|$.\par\smallskip
\begingroup\raggedright
\par\endgroup

\subsubsection*{A1b. Atomic, but not circuits}
Single critical value $\delta(E)/|E|$ with $\delta(E)\le-2$.\par\smallskip
\begingroup\raggedright
%
\par\endgroup

\subsubsection*{A1c. Molecular, but not atomic}
Single critical value $\delta(E)/|E|<0$, several blocks.\par\smallskip
\begingroup\raggedright
%
\par\endgroup

\subsubsection*{A1d. Not molecular}
Several negative critical values.\par\smallskip
\begingroup\raggedright
%
\par\endgroup

\subsection*{A2. Dependent BK-unions that are not cyclic}
All critical values are non-positive, and $0$ occurs. Together with the cyclic ones these are the dependent BK-unions.\par\smallskip
\begingroup\raggedright
%
\par\endgroup

\subsection*{A3. Dependent polymatroids that are not BK-unions}
A positive critical value occurs.\par\smallskip
\begingroup\raggedright
%
\par\endgroup

\section*{B. Independent polymatroids}
All critical values are non-negative.

\subsection*{B1. BK-polymatroids}
Molecular, with the single critical value $0$.

\subsubsection*{B1a. Irreducible BK-polymatroids}
Atomic.\par\smallskip
\begingroup\raggedright\everypar{}\clubpenalty=0 
%
\par\endgroup

\subsubsection*{B1b. Reducible BK-polymatroids}
Several blocks.\par\smallskip
\begingroup\raggedright\everypar{}\clubpenalty=0 
%
\par\endgroup

\subsection*{B2. Molecular independent polymatroids with positive critical value $\delta(E)/|E|>0$}

\subsubsection*{B2a. Essential independent polymatroids}
Atomic.\par\smallskip
\begingroup\raggedright\everypar{}\clubpenalty=0 
%
\par\endgroup

\subsubsection*{B2a$'$. Atomic independent polymatroids that are \emph{not} essential}
Some proper subset has defect $\delta(X)=\delta(E)>0$.\par\smallskip
\begingroup\raggedright\everypar{}\clubpenalty=0 
%
\par\endgroup

\subsubsection*{B2b. Molecular, but not atomic}
Several blocks.\par\smallskip
\begingroup\raggedright\everypar{}\clubpenalty=0 
%
\par\endgroup

\subsection*{B3. Independent polymatroids that are not molecular}
Several non-negative critical values.\par\smallskip
\begingroup\raggedright\everypar{}\clubpenalty=0 
%
\par\endgroup

\end{document}